\documentclass[12pt, twoside]{article}
\usepackage[a4paper,  left= 25mm, textheight=660pt]{geometry} 
\usepackage[new]{old-arrows}
\usepackage{mathtools}
\usepackage[utf8]{inputenc}
\usepackage{amsfonts}
\usepackage{amsmath}
\usepackage{amsthm}
\usepackage{amssymb}
\usepackage{hyperref}
\usepackage{mathrsfs}
\usepackage{tikz}
\usetikzlibrary {arrows.meta}
\usetikzlibrary{shapes.geometric}
\usepackage{tikz-cd}
\usepackage{float}
\usepackage[
backend=biber,
style=alphabetic,
maxnames=10, minalphanames=6, maxalphanames=10
]{biblatex} 
\usepackage[nottoc,notlot,notlof]{tocbibind}
\usepackage{titlesec}

\titleformat{\subsubsection}[runin]
{\normalfont\bfseries}{\thesubsubsection}{1em}{}

\newcommand{\ZZ}{\mathbb Z}

\newcommand{\QQ}{\mathbb Q}
\newcommand{\RR}{\mathbb R}
\newcommand{\GG}{\mathbb G}
\newcommand{\PP}{\mathbb P}

\newcommand{\NN}{\mathbb N}

\renewcommand{\AA}{\mathbb A}
\newcommand{\cO}{\mathcal O}
\newcommand{\cL}{\mathcal L}
\newcommand{\cM}{\mathcal M}

\newcommand{\cF}{\mathcal F}

\newcommand{\cK}{\mathcal K}
\newcommand{\cX}{\mathcal X}

\newcommand{\cY}{\mathcal Y}

\renewcommand{\Tilde}{\widetilde}
\renewcommand{\Bar}{\overline}

\DeclareMathOperator{\SL}{SL}

\DeclareMathOperator{\Spec}{Spec}
\DeclareMathOperator{\Hilb}{Hilb}

\DeclareMathOperator{\LW}{LW}
\DeclareMathOperator{\MR}{MR}

\DeclareMathOperator{\ind}{ind}

\DeclareMathOperator{\trop}{trop}

\DeclareMathOperator{\Sing}{Sing}

\DeclareMathOperator{\pr}{pr}

\DeclareMathOperator{\SWS}{SWS}

\DeclareMathOperator{\colim}{colim}

\DeclareMathOperator{\PSWS}{PSWS}
\DeclareMathOperator{\Ob}{Ob}

\DeclareMathOperator{\red}{red}
\DeclareMathOperator{\diag}{diag}

\numberwithin{equation}{subsection}

\newtheorem{theorem}{Theorem}[section]
\newtheorem*{theorem*}{Theorem}
\newtheorem{lemma}[theorem]{Lemma}
\newtheorem{proposition}[theorem]{Proposition}
\newtheorem{corollary}[theorem]{Corollary}
\theoremstyle{definition}
\newtheorem{definition}[theorem]{Definition}
\theoremstyle{remark}
\newtheorem{remark}[theorem]{Remark}

\begin{document}

\title{\Large{GOOD MODELS OF HILBERT SCHEMES OF POINTS OVER SEMISTABLE DEGENERATIONS}}
\author{CALLA TSCHANZ}
\date{}
\maketitle

\textbf{Abstract.}
In this paper, we explore different possible choices of expanded degenerations and define appropriate stability conditions in order to construct good degenerations of Hilbert schemes of points over semistable degenerations of surfaces, given as proper Deligne-Mumford stacks. These stacks provide explicit examples of constructions arising from the work of Maulik and Ranganathan.  This paper builds upon and generalises previous work in which we constructed a special example of such a stack. We also explain how these methods apply to constructing minimal models of type III degenerations of hyperk\"ahler varieties, namely Hilbert schemes of points on K3 surfaces.

\setcounter{tocdepth}{1}
\tableofcontents

\section{Introduction}

An important method in the study of smooth varieties is to allow these varieties to degenerate to a sum of simpler irreducible components and use the degeneration to recover information about the original variety.
Our aim is to explore various constructions of good degenerations of Hilbert schemes of points over a semistable family of surfaces $X\to C$. We combine techniques of expanded degenerations (or expansions), geometric invariant theory, stacks and tropical geometry, as well as observe a symmetry breaking phenomenon in the constructions which allows us to define explicit examples. This builds upon previous work in which we constructed such a degeneration of Hilbert schemes of points as a proper Deligne-Mumford stack (\cite{CT}). The stability conditions required to make this stack proper are much simpler than would be expected in general for this type of problem. This is an interesting consequence of the specific choices made in the expanded degenerations we considered. In this paper, we generalise this construction, investigate alternative choices of expanded degenerations and describe the appropriate refinements of the stability condition needed to obtain proper moduli stacks. We give a detailed analysis of the stability conditions needed to ensure universal closure and separatedness, both from a geometric and tropical perspective.

\subsubsection*{Setup and motivation.}
We recall the basic setup of this problem from \cite{CT}. Let $X\to C$ be a semistable projective family of surfaces over a smooth curve such that the singularities of the special fibre $X_0$ over $0\in C$ are given étale locally by three planes intersecting transversely in $\AA^3$. In other words, we may consider the étale local model for this family given by $\Spec k[x,y,z,t]/(xyz-t)$, where $k$ is an algebraically closed field of characteristic zero.

\medskip
Now, let $X^\circ \coloneqq X\setminus X_0$ and $C^\circ\coloneqq C\setminus \{0\}$ and consider the relative Hilbert scheme of $m$ points $\Hilb^m(X^\circ/C^\circ)$. Our goal is to construct compactifications of this object which satisfy some desired properties, which we discuss in the two bullet points below. The obvious compactification $\Hilb^m(X/C)$ does not satisfy these properties. In the case of the first bullet point, this comes from the fact that $\Hilb^m(X/C)$ is not logarithmically flat with respect to the pullback divisorial logarithmic structure given by the divisor $X_0$ in $X$. For the second point, it follows from the fact that the special fibre of $\Hilb^m(X/C)$ is very singular (which makes the total space very singular).

In $\Hilb^m(X^\circ/C^\circ)$, we have removed the limits of the families of length $m$ zero-dimensional subschemes and we wish to compactify, i.e.\ add limits back in in such a way that the resulting object satisfies “good" properties. These limits will again be length $m$ zero-dimensional subschemes, some on $X_0$ and some on different varieties. This will be achieved by constructing \emph{expansions} (birational modifications of $X_0$ in a 1-parameter family)  and selecting the appropriate limit for each family of length $m$ zero-dimensional subschemes in these modifications with the help of a stability condition. The two main meanings of “good" we may consider are the following.
\begin{itemize}
    \item The resulting degeneration of Hilbert schemes of points should be flat and proper over $C$, and the limit over $0\in C$ of each family of length $m$ zero-dimensional subschemes should be supported in the smooth locus of an underlying surface. Clearly, this is not the case for the obvious compactification $\Hilb^m(X/C)$, as certain length $m$ zero-dimensional subschemes in the fibre over $0\in C$ will have some of their support lying in $\Sing(X_0)$. We wish to remove these subschemes and replace them by subschemes having support in the smooth locus of some modifications of $X_0$. As in \cite{CT}, we will refer to the support of a subscheme lying in the smooth locus of a surface as the condition that the subscheme is \emph{smoothly supported}. Imposing these properties has the following useful consequences. The question of studying Hilbert schemes of $m$ points on $X_0$ is replaced by an easier problem. We study the Hilbert schemes of $m'$ points on the interior of each irreducible component of the modifications of $X_0$, where $m'\leq m$, and consider products of these simpler Hilbert schemes of points. The smooth fibres of the family of Hilbert schemes of $m$ points degenerate to a union of these products, which is easier to describe. Moreover, in such a compactification, the special fibre is stratified in a way which encodes the data of the families of length $m$ zero-dimensional subschemes. By observing where the limit of a family lands within this stratification, we may recover the defining equations of this family. These properties have applications to cohomological computations and enumerative geometry.
    \item Secondly, one can consider the more specific case where $X \to C$ is a type III semistable degeneration of K3 surfaces, with a relative holomorphic symplectic logarithmic 2-form. We can then try to construct a family of Hilbert schemes of points on $X$ which will be minimal in the sense of the minimal model program, and by this we mean a semistable or divisorial log terminal (dlt) degeneration with trivial canonical bundle (see Section \ref{minimality section} for details on minimality). The singularities arising in such a degeneration $X\to C$ are of the type described above, i.e.\ we can restrict ourselves to the local problem where $X_0$ is thought of as given by $xyz=0$ in $\AA^3$. Among other reasons, Hilbert schemes of points on K3 surfaces are interesting to study because they form a class of examples of hyperk\"ahler varieties.
\end{itemize}
In this paper, we will construct degenerations of Hilbert schemes of points satisfying the properties of the first point above. We then show that the solutions considered here also give us constructions of semistable minimal hyperk\"ahler degenerations as stacks. It is worth noting, however, that it is possible to construct dlt minimal degenerations as schemes which do not fulfill the requirement of the first point.

\subsubsection*{Main results.} Here, we briefly describe the main results of this paper.

\medskip
We start by making a sequence of blow-ups on a fibre product $X\times_{\AA^1}\AA^{n+1}$ over a base $C\times_{\AA^1}\AA^{n+1}$. This process is similar to the construction in \cite{CT} (see Section \ref{previous construction section} for an overview of this previous construction), however, here, we allow for more types of blow-ups to be made. The sequence of blow-ups on $X\times_{\AA^1}\AA^{n+1}$ is encoded by a choice of sets $A,B\subset \{1,\dots, n+1\}$ which satisfy a specific property called the \emph{unbroken} property (see Definition \ref{unbroken}). We name the resulting expanded degenerations $X[A,B]$. For a specific choice of $A$ and $B$, we recover the expanded degeneration $X[n]$ of \cite{CT}. We describe an action of $\GG_m^n$ on $X[A,B]$, an embedding of $X[A,B]$ into a product of projective bundles and discuss various GIT stability conditions on the relative Hilbert scheme of $m$ points of this expanded degeneration $\Hilb^m(X[A,B]/C\times_{\AA^1}\AA^{n+1})$. We then present a stack of expansions $\mathfrak{C}'$ with a family $\mathfrak{X}'$ over it and explore the different choices of stability condition which will allow us to construct a stack of stable length $m$ zero-dimensional subschemes in $\mathfrak{X}'$ which is proper over $C$. We start by extending the LW and SWS stability conditions defined in \cite{CT} to this setting and denote by $\mathfrak{N}^m_{\LW}$ and $\mathfrak{N}^m_{\SWS}$ the corresponding stacks of LW and SWS stable length $m$ zero-dimensional subschemes on $\mathfrak{X}'$. We show that, though both of these stacks are universally closed, they are not separated, as they contain several non-equivalent representatives for the same limit. We consider two main approaches to solve this problem.

\medskip
The first method to construct separated stacks is to define equivalence classes identifying all limits of a given object in $\mathfrak{N}^m_{\LW}$ and $\mathfrak{N}^m_{\SWS}$. The stacks whose objects are these equivalence classes are denoted $\Bar{\mathfrak{N}^m_{\LW}}$ and $\Bar{\mathfrak{N}^m_{\SWS}}$. These are no longer algebraic stacks but we may now prove the following theorem.

\begin{theorem}
    The stacks $\Bar{\mathfrak{N}^m_{\LW}}$ and $\Bar{\mathfrak{N}^m_{\SWS}}$ are proper over $C$ and have finite automorphisms.
\end{theorem}

This approach parallels work of Kennedy-Hunt on logarithmic Quot schemes \cite{K-H}.

\medskip
The second method we use is to define an additional stability condition, called \textit{$(\alpha,\beta)$-stability}, which we use to cut out proper substacks of $\mathfrak{N}^m_{\LW}$ and $\mathfrak{N}^m_{\SWS}$. The $(\alpha,\beta)$-stability condition must satisfy some subtle properties in order for the substack it cuts out to be proper over $C$: it must be what we call a \textit{proper LW or SWS stability condition} (see Definition \ref{proper SWS stab}). A proper LW stability condition $(\alpha,\beta)$ as defined here recovers one of the choices of stability condition arising from the methods of Maulik and Ranganathan \cite{MR} in the case of Hilbert schemes of points. We denote by $\mathfrak{N}^m_{\MR,(\alpha,\beta)}$ and $\mathfrak{N}^m_{\PSWS,(\alpha,\beta)}$ the stacks of stable length $m$ zero-dimensional subschemes in $\mathfrak{X}'$ for proper LW and SWS stability conditions $(\alpha,\beta)$. We are then able to prove the following results.

\begin{theorem}
    The stacks $\mathfrak{N}^m_{\MR,(\alpha,\beta)}$ and $\mathfrak{N}^m_{\PSWS,(\alpha,\beta)}$ are Deligne-Mumford and proper over $C$.
\end{theorem}

\begin{theorem}
    For any $(\alpha,\beta)$ which defines a proper SWS stability condition, there is an isomorphism of stacks
    \[
    \mathfrak{N}^m_{\MR,(\alpha,\beta)}\cong \mathfrak{N}^m_{\PSWS,(\alpha,\beta)}.
    \]
\end{theorem}

We would like to emphasize here that it is not true that for every stack $\mathfrak{N}^m_{\MR,(\alpha,\beta)}$ there exists an equivalent stack $\mathfrak{N}^m_{\PSWS,(\alpha,\beta)}$.

\medskip
These choices of proper substack distinguish some geometrically meaningful examples among all possible choices arising from the methods of Maulik and Ranganathan. Our examples do not recover every possible choice, because the geometric assumptions made in the construction of $\mathfrak{X'}$ mean that it does not contain all possible expansions of $X$. The reason for these assumptions is discussed in Section \ref{2nd BUs}. The stacks $\mathfrak{M}^m_{\LW}$ and $\mathfrak{M}^m_{\SWS}$ of \cite{CT} are special cases of $\mathfrak{N}^m_{\MR,(\alpha,\beta)}$ and $\mathfrak{N}^m_{\PSWS,(\alpha,\beta)}$ stacks.

\medskip
An important concern in the construction of these proper stacks is to understand what smoothings exist from one special element of the degeneration to another. Indeed, we must ensure that not only the generic elements of a family have unique limits, but also that special elements of these families with support in modifications of $X_0$ have unique limits as well. In order to show that this is the case, it is necessary to study what smoothings exist from one modification of $X_0$ to another. We give a tropical criterion to explain how to understand these smoothings in Section \ref{smoothings section}.

\medskip
Finally, we are able to show that all proper algebraic stacks constructed here provide good minimal models (see Section \ref{minimality section} for a precise definition of the term good minimal model). In particular, we prove the following result.

\begin{theorem}
    If $X\to C$ is a semistable type III degeneration of K3 surfaces equipped with a relative holomorphic symplectic logarithmic 2-form, then this 2-form induces a nowhere degenerate relative logarithmic 2-form on each of the constructions presented above.
\end{theorem}

\subsubsection*{Background.}
Previous work in this area includes that of Li \cite{Li}, in which the method of expanded degenerations is introduced, and Li and Wu \cite{LW}, where it is used to construct modular compactifications for degenerations of Quot schemes relative to simple normal crossing boundaries with smooth singular locus. Following on from this, Gulbrandsen, Halle and Hulek \cite{GHH} present a GIT version in the case of Hilbert schemes of points, and show the stack quotient of their GIT quotient to be equivalent to the proper Deligne-Mumford stack of Li and Wu. These solutions provide a foundation for our work, but are limited by the requirement that the special fibre of $X\to C$ should have smooth singular locus. This is quite a serious restriction, as not many families will satisfy this condition. Note that the type of singularity we consider in this work is specifically the type that \cite{LW} and \cite{GHH} do not address.

In more recent work, Maulik and Ranganathan \cite{MR} use these techniques of expansions in the world of logarithmic and tropical geometry to solve this problem for Hilbert schemes over $X$ where $X_0$ may have any type of simple normal crossing singularity. This is a significant improvement, as their solution can be applied to any semistable family of surfaces $X$ and these can always be obtained by semistable reduction. Their construction, however, yields an infinite family of birational solutions, but no explicit model. Moreover, it is far from obvious how such an explicit model should be constructed. This is the problem we solve in \cite{CT} and in this paper.

\subsubsection*{Organisation.}
We will start by giving a brief overview of the results of \cite{CT} and by setting up the tropical point if view in Section \ref{background section}. In Section \ref{2nd construction}, we construct the expanded degenerations $X[A,B]$ and extend the group action, line bundle and GIT discussions from \cite{CT} to this context. Then, in Section \ref{stack section} we define the stack of expansions $\mathfrak{C}'$, the family $\mathfrak{X}'$ over $\mathfrak{C}'$. We describe how LW and SWS stability extend to this context, define the corresponding stacks $\mathfrak{N}^m_{\LW}$ and $\mathfrak{N}^m_{\SWS}$ of stable subschemes on $\mathfrak{X}'$ and show that these are universally closed but not separated. In Section \ref{separated stack section}, we discuss different methods for constructing separated stacks, as well as a tropical criterion for understanding smoothings between modifications of $X_0$. Finally, we discuss how all these constructions relate to the minimal model program in Section \ref{minimality section}.

\subsubsection*{Acknowledgements.}
I am grateful to Gregory Sankaran for his support and encouragement during this project. I would also like to thank Dhruv Ranganathan, Alastair Craw, Patrick Kennedy-Hunt and Thibault Poiret for all their help, as well as Siao Chi Mok and Terry Song for many interesting conversations. I am also grateful to Christian Lehn, Martin Ulirsch, Grégoire Menet and Grzegorz Kapustka for their useful suggestions. This work was undertaken while funded by the University of Bath Research Studentship Award and completed while funded under the grant Narodowe Centrum Nauki 2018/30/E/ST1/00530 at the Jagiellonian University in Krakow.

\section{Background}\label{background section}

In this section we recall the main results of \cite{CT} and the setup of the tropical perspective which we will refer to throughout this paper.

\subsubsection*{The semistable family of surfaces.}
Let $X\to C$ be a flat projective semistable family of surfaces over a curve isomorphic to $\AA^1$, and we consider an étale local model $\Spec k [x,y,z,t]/(xyz-t)\to \Spec k[t]$ for this family, where $k$ is an algebraically closed field of characteristic zero. We denote by $X_0$ the special fibre over $0\in C$ and by $Y_1$, $Y_2$ and $Y_3$ the irreducible components of this special fibre given étale locally by $x=0, y=0$ and $z=0$ respectively.

\subsection{The previous construction}\label{previous construction section}

In \cite{CT}, we start by constructing expanded degenerations $X[n]$ given by a sequence of blow-ups on a fibre product $X\times_{\AA^1}\AA^{n+1}$ over an expanded base $C[n]\coloneqq C\times_{\AA^1}\AA^{n+1}$. This construction is similar to those presented in \cite{LW} and \cite{GHH}, but additional blow-ups are added, breaking the symmetry of the construction, to allow for all necessary modifications for our problem to be included, while preserving commutativity of the blow-up morphisms. This family $X[n]$ comes with a natural torus action by the group $\GG_m^n$ and is shown to embed into a product of projective bundles. This allows us to define various $\GG_m^n$-linearised ample line bundles on the relative Hilbert scheme of $m$ points of this expanded degeneration $\Hilb^m(X[n]/C[n])$ and describe the corresponding GIT stable loci.

We then construct a stack of expansions $\mathfrak{C}$ and family $\mathfrak{X}$ over it. This is a substack of the stack $\mathfrak{X}'$ defined in Section \ref{family over stack of expansions}. We describe how \textit{Li-Wu stability} (abbreviated LW stability) can be extended to this setting and define an alternative notion of stability, called \textit{smoothly supported weak strict stability} (abbreviated SWS stability), derived from GIT stability conditions. We then construct the stacks $\mathfrak{M}^m_{\LW}$ and $\mathfrak{M}^m_{\SWS}$ of LW and SWS stable length $m$ zero-dimensional subschemes on $\mathfrak{X}$. The main results of \cite{CT} are the following.
\begin{theorem}
    The stacks $\mathfrak{M}^m_{\LW}$ and $\mathfrak{M}^m_{\SWS}$ are Deligne-Mumford and proper.
\end{theorem}

\begin{theorem}
    There is an isomorphism of stacks
    \[
    \mathfrak{M}^m_{\LW}\cong \mathfrak{M}^m_{\SWS}.
    \]
\end{theorem}
Crucially, the choices of blow-ups allowed in the construction of $X[n]$ are very restricted. This has the surprising consequence that the stacks of LW and SWS stable subschemes $\mathfrak{M}^m_{\LW}$ and $\mathfrak{M}^m_{\SWS}$ are already proper without the need to add any additional stability condition. The work of Maulik and Ranganathan \cite{MR} predicts that in general this will not be the case. Indeed, we will see in Section \ref{separated stack section} that when allowing for different choices of expanded degenerations we will need to introduce an additional stability condition in order to obtain separated stacks over $C$. This stability condition is analogous to the Donaldson-Thomas stability of \cite{MR}.

\subsection{Tropical perspective}

Throughout this work, the results will be presented from a geometric and tropical point of view. Indeed, the key property that we would like our compactifications of $\Hilb^m(X^\circ/C^\circ)$ to possess can be thought of as the property that the limits of the families of length $m$ zero-dimensional subschemes in the compactification should be stratified in a way which records the data of these families. In étale local coordinates, this is essentially asking for the compactifications to encode the degrees of vanishing of these families in the variables $x,y$ and $z$. Tropical geometry is a natural tool to use to study such questions, as it is precisely the object parameterising such data.

\medskip
We recall here the essential details of the tropical perspective (see \cite{CT} and \cite{MR} for more details). We construct the tropicalisation of $X$ with respect to the divisorial logarithmic structure given by the divisor $X_0$. Formally, this means that we define a sheaf of monoids
\[
\cM_X(U)\coloneqq \{f\in \cO_X(U) \ | \ f|_{U\setminus X_0} \in \cO^*_X(U\setminus X_0)\},
\]
for any open set $U\subseteq X$, and the corresponding \emph{characteristic sheaf}
\(
\Bar{\cM}_X \coloneqq \cM_X/\cO^*_X.
\)
The \emph{tropicalisation} of $X$ is then given by $\trop(X)\coloneqq \colim_{x\in X}(\Bar{\cM}_{X,x})^\vee$. We recall also the following definition.

\begin{definition}
    Let $\Upsilon$ be a fan, let $|\Upsilon|$ be its support and $\upsilon$ be a continuous map
    \[
    \upsilon\colon |\Upsilon| \longrightarrow \Sigma_X
    \]
    such that the image of every cone in $\Upsilon$ is contained in a cone of $\Sigma_X$ and that is given by an integral linear map when restricted to each cone in $\Upsilon$. We say that $\upsilon$ is a \textit{subdivision} if it is injective on the support of $\Upsilon$ and the integral points of the image of each cone $\tau\in\Upsilon$ are exactly the intersection of the integral points of $\Sigma_X$ with $\tau$.
\end{definition}

\begin{definition}
    The subdivision
\(
\Upsilon \longhookrightarrow \Sigma_X \longhookrightarrow \RR^r_{\geq 0}
\)
has an associated toric variety $\AA_{\Upsilon}$, which comes with a $\GG^r_m$-equivariant birational map $\AA_\Upsilon \to \AA^r$. There is then an induced morphism of quotient stacks
\(
[\AA_\Upsilon/\GG_m^r] \longrightarrow [\AA^r/\GG_m^r].
\)
We define the \textit{expansion} of $X$ associated to such a subdivision $\Upsilon$ to be the modification
\[
X_\Upsilon \coloneqq X \times_{[\AA^r/\GG_m^r]} [\AA_\Upsilon/\GG_m^r].
\]
\end{definition}

In the étale local model which we consider, the functions vanishing at $X_0$ are $x,y$ and $z$, therefore we may represent $\trop(X)$ as a fan in $\RR_{\geq 0}^3$ given by the positive orthant and its faces. For convenience, we will 
denote by $\trop(X_0)$ the hyperplane slice through $\trop(X)$ resulting from fixing a height in this three-dimensional cone. Recall from \cite{CT} that this choice of height is equivalent to a choice of point in the half-line corresponding to the tropicalisation of $\AA^1$ (where this tropicalisation is taken with respect to the divisor $0\in \AA^1$). Changing the height corresponds to making some base change on the family $X$.
Figure \ref{geom and trop special fibre} shows a copy of the special fibre $X_0$ both from the geometric point of view, on the left, and tropical point of view, on the right. We will abuse notation slightly and label the vertices, edges and interior of the triangle $\trop(X_0)$ by $Y_i$, $Y_i\cap Y_j$  and $Y_1\cap Y_2\cap Y_3$ again for convenience, as shown in Figure \ref{geom and trop special fibre}. Subdivisions of the tropicalisation will be represented by adding edges and vertices to this triangle, as we will see in the following section.

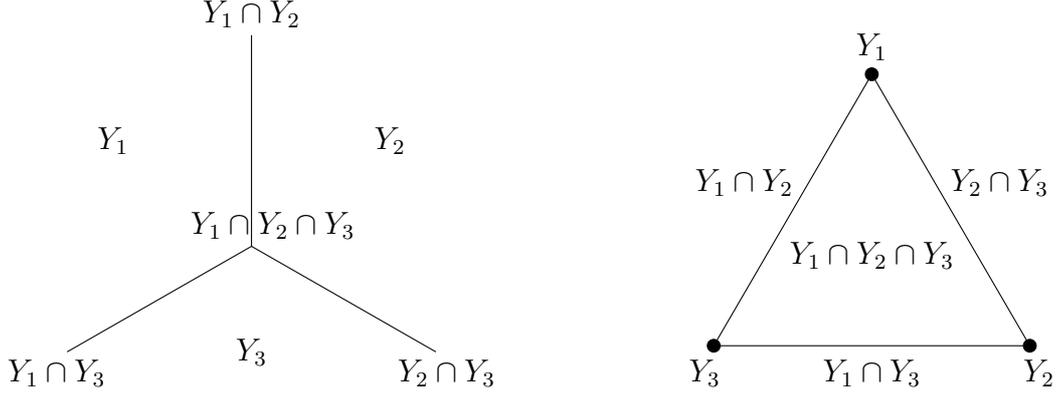
\begin{figure} 
    \begin{center}   
    \begin{tikzpicture}[scale=1.4]
        \draw   
        (0,0) -- (0,2)       
        (-1.732, -1) -- (0,0)
        
        (1.732, -1) -- (0,0);
        \draw (-1.3,1) node[anchor=center]{$Y_1$};
        \draw (1.3,1) node[anchor=center]{$Y_2$};
        \draw (0,-1) node[anchor=center]{$Y_3$};
        \draw (0,2.2) node[anchor=center]{$Y_1\cap Y_2$};
        \draw (-1.832, -1.2) node[anchor=center]{$Y_1\cap Y_3$};
        \draw (1.832, -1.2) node[anchor=center]{$Y_2\cap Y_3$};
        \draw (0.2, 0.2) node[anchor=center]{$Y_1\cap Y_2 \cap Y_3$};
    \end{tikzpicture}
    \hspace{2cm}
    \begin{tikzpicture}[scale=1.2]
\draw   

        (-1.732, -1) -- (0,2)       
        (-1.732, -1) -- (1.732, -1)
        
        (1.732, -1) -- (0,2);
        \draw (-1.4,0.8) node[anchor=center]{$Y_1\cap Y_2$};
        \draw (1.4,0.8) node[anchor=center]{$Y_2\cap Y_3$};
        \draw (0,-1.3) node[anchor=center]{$Y_1\cap Y_3$};
        \filldraw[black] (0,2) circle (2pt) ;
        \filldraw[black] (-1.732, -1) circle (2pt) ;
        \filldraw[black] (1.732, -1) circle (2pt) ;
        \draw (0,2.3) node[anchor=center]{$Y_1$};
        \draw (-1.832, -1.3) node[anchor=center]{$Y_3$};
        \draw (1.832, -1.3) node[anchor=center]{$Y_2$};
        \draw (0, 0) node[anchor=center]{$Y_1\cap Y_2 \cap Y_3$};
\end{tikzpicture}
    \end{center}
    \caption{Geometric and tropical pictures of the special fibre $X_0$.}
    \label{geom and trop special fibre}
\end{figure}

Finally, we consider the \emph{tropicalisation map}, as defined in \cite{MR}. This is a morphism
\[
\trop \colon X^\circ \longrightarrow \Sigma_X.
\]
For an open subscheme $Z^\circ\subset X^\circ$, we denote by $\trop(Z^\circ)$ the image of the map $\trop$ restricted to $Z^\circ(\cK)$, where $\mathcal{K}$ is a valued field extending $k$. This image $\trop(Z^\circ)$ can be seen as a collection of rays in $\trop(X)$ or vertices in $\trop(X_0)$. This is the basic data around which the expansions are built.

\section{The expanded constructions}\label{2nd construction}
\subsection{Scheme construction}\label{2nd BUs}

We construct expanded degenerations as schemes $X[A,B]$ depending on a choice of sets $A$ and $B$. This forms a generalisation of the schemes $X[n]$ presented in \cite{CT}. We recall some of the notation and terminology from \cite{CT} here for convenience.

\subsubsection*{Enlarging the base.} Take a copy of $\AA^{n+1}$, with elements labelled $(t_1, \ldots, t_{n+1}) \in \AA^{n+1}.$ To match our previous terminology, we refer to the entries $t_i$ as \textit{basis directions}. We then take a fibre product $X\times_{\AA^1} \AA^{n+1}$ given by the map $X\to C\cong \AA^1$ and the product
\[
(t_1, \ldots, t_{n+1}) \longmapsto t_1\cdots t_{n+1}.
\]
This fibre product contains several copies of the special fibre $X_0$. We will now be able to make different modifications of this fibre by taking a sequence of blow-ups on $X\times_{\AA^1} \AA^{n+1}$ along $Y_1$ and $Y_2$ and the vanishing of some basis directions. Unlike the construction of \cite{CT}, we do not require here that both $Y_1$ and $Y_2$ should be blown up along all basis directions. Instead, the choice of blow-ups will be dictated by a choice of sets $A$ and $B$.

\subsubsection*{The sets $A$ and $B$.}
For geometric reasons explained in Remark \ref{remark about unbroken condition}, we will consider blow-ups arising only from sets $A$ and $B$ satisfying the following condition.

\begin{definition}\label{unbroken}
    Let $A$ and $B$ be subsets of $[n+1]\coloneqq \{1,\dots,n+1\}$. We say that the pair $(A,B)$ is \emph{unbroken} if there are real intervals $[1,a)$ and $(b,n+1]$ with $1<b<a<n+1$ such that $A = \NN\cap [1,a)$ and $B= \NN \cap (b,n+1]$. We will often write $(A,B,n)$ for convenience when discussing these unbroken pairs; this is done in order to keep track of the value of $n$, although this value is determined by the sets $A$ and $B$.
\end{definition}

In other words, if $(A,B)$ is an unbroken pair, then elements of $[n+1]$ in order from 1 to $n+1$ are contained first in $A\setminus B$, then in $A\cap B$ and then in $B\setminus A$, where $A\cap B$ may be empty but the other two are not. In particular, $A$ is forced to contain 1 but not $n+1$ and, similarly, $B$ must contain $n+1$ but not 1. See Remark \ref{remark unbroken trop} for a tropical description of the unbroken condition.

\subsubsection*{The blow-ups.}
Let $(A,B,n)$ be an unbroken pair. In the étale local model, the blow-ups on $X\times_{\AA^1}\AA^{n+1}$ are expressed as blow-ups in the ideals
\begin{align*}
    \langle x,t_{1} \rangle, \langle x,t_{1}t_{2} \rangle ,\dots ,\langle x,t_{1}t_{2}\cdots t_{\lfloor a \rfloor} \rangle
\end{align*}
and the ideals
\begin{align*}
    \langle y,t_{n+1} \rangle, \langle y,t_{n+1}t_{n} \rangle ,\dots ,\langle y,t_{n+1}t_{n}\cdots t_{\lceil b \rceil} \rangle.
\end{align*}
Globally on $X$, these are blow-ups along the vanishing equations of the components $Y_1$ and $Y_2$. The equations of the blow-ups in étale local coordinates are given by
\begin{align}\label{BU1 eqns}
    &x_0^{(1)}t_1 = xx_1^{(1)}, \nonumber \\
    &x_1^{(k-1)}x_0^{(k)}t_k = x_0^{(k-1)}x_1^{(k)}, \qquad \textrm{ for } \ 2\leq k\leq n, \nonumber \\
    &x_0^{(\lfloor a \rfloor)} yz = x_1^{(\lfloor a \rfloor)} t_{\lfloor a \rfloor+1} \cdots t_{n+1}. \\
    &y_0^{(1)}t_{n+1} = yy_1^{(1)}, \nonumber \\
    &y_1^{(k-1)}y_0^{(k)}t_{n+2-k} = y_0^{(k-1)}y_1^{(k)} \ \textrm{ for } \ 2\leq k\leq n, \nonumber \\
    &y_0^{(\lceil b \rceil)} xz = y_1^{(\lceil b \rceil)} t_{1} \cdots t_{\lceil b \rceil-1} \nonumber 
\end{align}
and we get
\[
x_0^{(k)} y_0^{(n+1-k)} z = x_1^{(k)} y_1^{(n+1-k)}
\]
when $k\in A$ and $k+1\in B$.

\medskip
\emph{Expanded degeneration notation.} We denote the resulting space after making these blow-ups by $X[A,B]$ and write $C[A,B]$ to refer to the base $C\times_{\AA^1}\AA^{n+1}$ together with the data of the sets $A$ and $B$. This is the same scheme as $C[n]$ from \cite{CT}, but here, for each basis direction $t_i$, we retain the additional information of whether $i$ belongs to the set $A$ or $B$. Let
\[b\colon X[A,B] \longrightarrow X\times_{\AA^1} \AA^{n+1}\]
denote the sequence of blow-ups and
\[
\pi \colon X[A,B] \longrightarrow X
\]
denote the natural projection.
If $A=\{ 1,\dots, n \}$ and $B=\{ 2,\dots, n+1 \}$, then we get back exactly the space $X[n]$ of \cite{CT}.

\medskip
We fix the notation $[k]\coloneqq \{1,\dots, k\}$. Let $|A| = l_A$, $|B| = l_B$ and let
\[
\iota_A\colon [l_A] \longrightarrow A \quad \textrm{and} \quad \iota_B\colon [l_B] \longrightarrow B
\]
be the order preserving morphisms. Then for $i\leq l_A$ and $j\leq l_B$, we define $X_{(i,j)}$ to be the space resulting from having blown up $X\times_{\AA^1} \AA^{n+1}$ along the pullback of $Y_1$ and the vanishing of $t_{\iota_A(k)}$ for all $k\in [i]$, and the pullback of $Y_2$ and the vanishing of $t_{\iota_B(k)}$ for all $k\in [j]$. For now, we may assume that the blow-ups were made in that order, though by the second point of Proposition \ref{properties of BUs} this order does not actually affect the outcome. Let
\begin{align*}
    &\beta^A_{(i,j)} \colon X_{(i,j)} \longrightarrow X_{(i-1,j)}, \\
    &\beta^B_{(i,j)} \colon X_{(i,j)} \longrightarrow X_{(i,j-1)},
\end{align*}
denote the morphisms corresponding to each individual blow-up. We therefore have the equality
\[
\beta^B_{(l_A,l_B)}\circ \cdots \circ \beta^B_{(l_A,1)}\circ \beta^A_{(l_A,0)}\circ \cdots \circ \beta^A_{(1,0)} = b.
\]

\begin{proposition}\label{properties of BUs}
    The expanded degeneration $X[A,B] \to C[A,B]$ satisfies the following properties.
    \begin{enumerate}
        \item The local model for $X[A,B]$ embeds into $(X\times_{\AA^1} \AA^{n+1})\times (\PP^1)^{\lfloor a \rfloor + \lceil b \rceil}$.
        \item The order of the blow-up morphisms $\beta^A_{(i,j)}$ and $\beta^B_{(i,j)}$ commutes.
        \item The morphism $X[A,B] \to C[A,B]$ is projective.
    \end{enumerate}
\end{proposition}

\begin{proof}
    The first point is immediate from the construction. The second point follows from Proposition 3.1.5 of \cite{CT} and the third point follows from Proposition 3.1.7 of \cite{CT}.
\end{proof}

\subsubsection*{Setting terminology.} Here, for convenience, we recall some terminology from \cite{CT}.

\begin{definition}
    We say that a dimension 2 component in a fibre of $X[A,B] \to C[A,B]$ is a $\Delta_1$-\textit{component} if it is contracted by the morphism $\beta^A_{(i,l_B)}$ for some $i\leq l_A$. Moreover if a $\Delta_1$-component in a fibre is contracted by such a map then we say it is \textit{expanded out} in this fibre. We label by $\Delta_1^{(i)}$ the $\Delta_1$-component resulting from the $i$-th blow-up along $Y_1$. Similarly, a $\Delta_2$-\textit{component} is a component which is contracted by a morphism $\beta^B_{(l_A,j)}$ for some $j\leq l_B$ and the $\Delta_2$-component resulting from this $j$-th blow-up along $Y_2$ is denoted $\Delta_2^{(j)}$.
    
    We say that a dimension 2 component in a fibre of $X[A,B]\to C[A,B]$ is a $\Delta$-\textit{component} if it is a $\Delta_i$-component for some $i$. If it is expanded out in some fibre we may alternatively refer to it as an \textit{expanded component}. We say that a $\Delta$-component is \textit{equal to} a component $W$ of a fibre of $X[A,B]$ if the projective coordinates associated to this $\Delta$-component are proportional to the non-vanishing coordinates of $W$.
\end{definition}

In the étale local notation, the $\Delta_1^{(i)}$ component is introduced by the blow-up of the ideal $\langle x,t_{1}t_{2}\cdots t_{i} \rangle$ and the $\Delta_2^{(n+2-j)}$ component is introduced by the blow-up of the ideal $\langle y,t_{n+1}t_{n}\cdots t_{j} \rangle$.

\begin{definition}
    We refer to an irreducible component of a $\Delta$-component as a \textit{bubble}. The notions of two bubbles being \textit{equal} and a bubble being \textit{expanded out} in a certain fibre follow directly from our previous definitions.
\end{definition}

\begin{definition}\label{delta multiplicity}
    We say that a $\Delta_i$-component is of \textit{pure type} if it is not equal to any $\Delta_j$-component for $j\neq i$. Otherwise we say it is of \textit{mixed type}. We will say a component is \emph{pure of type $i$} to say it is a $\Delta_i$-component of pure type. Similarly we will refer to edges in $\trop(X_0)$ parallel to $Y_2\cap Y_3$ as being of \emph{type 1} and edges parallel to $Y_1\cap Y_3$ as being of \emph{type 2}.

    A $\Delta$-component in a fibre of $X[A,B]$ is said to have $\Delta_i$-\textit{multiplicity} $l$ if this component is equal to $\Delta_i^{(j)}= \dots = \Delta_i^{(j+l)}$ for some $j$ and it is not equal to any other $\Delta_i$-components.
\end{definition}

Note that, in the setting of the construction $X[n]$ of \cite{CT}, if both the $\Delta_1$- and $\Delta_2$-multiplicities of a component are nonzero, then they must be identical.

\begin{definition}\label{base codimension}
    We say that a fibre in some expanded degeneration $X[A,B]$ has \textit{base codimension} $k$ if exactly $k$ basis directions vanish at this fibre. This is independent of the value $n$.
\end{definition}

\subsubsection*{Visualising the construction.} We will now describe how to visualise the fibres of $X[A,B]\to C[A,B]$. In the following, let $(t_1,\dots,t_{n+1})\in C[A,B]$ and let us assume that $t_i$ and $t_j$ are two consecutive zero entries of $(t_1,\dots,t_{n+1})$, i.e.\ $t_i=t_j=0$ and $t_{i+1},\dots, t_{j-1} \neq 0$. 

\medskip
\emph{Expanded components of pure type.} If both $i$ and $j$ are in $A$ and $j\notin B$, then the component $\Delta_1^{(i)}= \dots = \Delta_1^{(j-1)}$ is expanded out in the fibre of $X[A,B]$ over $(t_1,\dots,t_{n+1})$ and it is of pure type. If all other basis directions are nonzero in this fibre, then $Y_1 = \Delta_1^{(j)} = \dots = \Delta_1^{(\lfloor a \rfloor)}$ and $Y_2\cup Y_3 = \Delta_1^{(i-1)} = \dots = \Delta_1^{(1)}$. There is no component $\Delta_2^{n+2-k}$ in the whole of $X[A,B]$ for $k\leq j$ and all $\Delta_2$-components in this fibre are equal to $Y_1\cup Y_3$. Figure \ref{figure pure 1 type component} shows such a fibre. On the right hand side of the figure, we see how to represent this fibre tropically, where the added red vertices are the two bubbles forming the expanded out $\Delta_1$-component and the edge between them is the intersection of these bubbles. Similarly, if both $i$ and $j$ are in $B$ and $i\notin A$, then the component $\Delta_2^{(n+1-i)}= \dots = \Delta_2^{(n+2-j)}$ is expanded out in the fibre of $X[A,B]$ over $(t_1,\dots,t_{n+1})$ and of pure type.

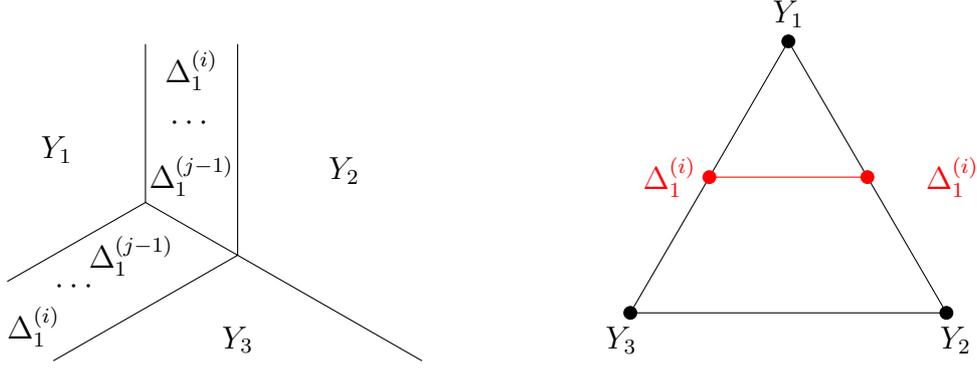
\begin{figure} 
    \begin{center}   
    \begin{tikzpicture}[scale=1.4]
        \draw   (-0.866,0.5) -- (0,0)
        (0,0) -- (0,2)       
        (-0.866,0.5) -- (-0.866,2)
        (-1.732, -1) -- (0,0)
        (-2.165,-0.249) -- (-0.866,0.5)
        
        (1.732, -1) -- (0,0);
        \draw (-1.916,-0.675) node[anchor=center]{$\Delta_1^{(i)}$};
        \draw (-1.516,-0.3) node[anchor=center]{$\cdots$};
        \draw (-1,0) node[anchor=center]{$\Delta_1^{(j-1)}$};
        \draw (-0.433,1.75) node[anchor=center]{$\Delta_1^{(i)}$};
        \draw (-0.433,1.25) node[anchor=center]{$\cdots $};
        \draw (-0.433,0.75) node[anchor=center]{$\Delta_1^{(j-1)}$};
        \draw (-1.7,1) node[anchor=center]{$Y_1$};
        \draw (1,0.8) node[anchor=center]{$Y_2$};
        \draw (0,-0.8) node[anchor=center]{$Y_3$};
        
    \end{tikzpicture}
    \hspace{2cm}
    \begin{tikzpicture}[scale=1.2]
\draw   

        (-1.732, -1) -- (0,2)       
        (-1.732, -1) -- (1.732, -1)
        
        (1.732, -1) -- (0,2)
        
        ;
        \draw[red] (-0.866,0.5) -- (0.866,0.5);
        \draw (-1.3,0.5) node[anchor=center, color = red]{$\Delta_1^{(i)}$};
        \draw (1.8,0.5) node[anchor=center, color = red]{$\Delta_1^{(i)}$};
        \filldraw[red] (-0.866,0.5) circle (2pt) ;
        \filldraw[red] (0.866,0.5) circle (2pt) ;
        \filldraw[black] (0,2) circle (2pt) ;
        \filldraw[black] (-1.732, -1) circle (2pt) ;
        \filldraw[black] (1.732, -1) circle (2pt) ;
        \draw (0,2.3) node[anchor=center]{$Y_1$};
        \draw (-1.832, -1.3) node[anchor=center]{$Y_3$};
        \draw (1.832, -1.3) node[anchor=center]{$Y_2$};
\end{tikzpicture}
    \end{center}
    \caption{Geometric and tropical picture at $t_i=t_j=0$ for $i,j\in A\setminus B$ in $X[A,B]$.}
    \label{figure pure 1 type component}
\end{figure}

\medskip
\emph{Expanded components of mixed type.} Let $t_i$ and $t_j$ be two consecutive zero entries of $(t_1,\dots,t_{n+1})$ as before. If $i\in A$ and $j\in B$, then $\Delta_1^{(i)}=  \Delta_1^{(n+2-j)}$ is expanded out in the fibre of $X[A,B]$ over $(t_1,\dots,t_{n+1})$. This situation is similar to that discussed in \cite{CT}, namely there is a bubble of mixed type in the $\pi^*(Y_1\cap Y_2)$ locus of such a fibre and a bubble of pure type in each of the loci $\pi^*(Y_1\cap Y_3)$ and $\pi^*(Y_2\cap Y_3)$. This is true regardless of whether $i$ is in $B$ or whether $j$ is in $A$, however this information affects the $\Delta$-multiplicity of the $Y_k$ components.

\medskip
\emph{Fibres with both components of mixed and pure type.} It is of course also possible to have fibres where both types of components described above appear. For example, let $(t_1,t_2,t_3) = (0,0,0)$, where $A\coloneqq \{1,2\}$ and $B \coloneqq \{3\}$. Then the fibre of $X[A,B]$ over this point has an expanded component $\Delta_1^{(1)}$ of pure type, as well as expanded components $\Delta_1^{(2)}$ and $\Delta_2^{(1)}$ meeting in a bubble of mixed type in the $\pi*(Y_1\cap Y_2)$ locus and separating out into two bubbles of pure type in the $\pi*(Y_1\cap Y_3)$ and $\pi*(Y_2\cap Y_3)$ loci. This is depicted in Figure \ref{mixed and pure}.

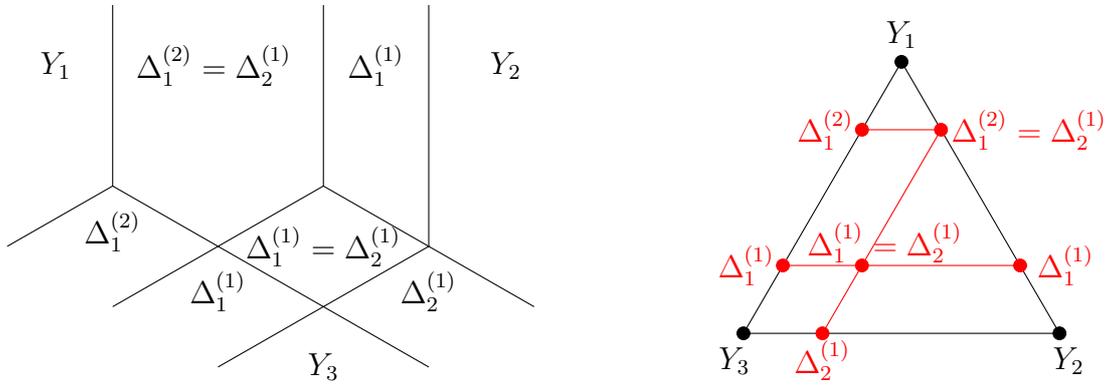
\begin{figure} 
    \begin{center}   
    \begin{tikzpicture}[scale=1.6]
        \draw    (0.866,-0.5) -- (0,-1)
        (0,-1) -- (-0.866,-0.5)
        (0,0) -- (0,1.5)
        (0,0) -- (-0.866,-0.5)
        (0,0) -- (0.866,-0.5)
        (-0.866,-0.5) -- (-1.732,0)       
        (-1.732,0) -- (-1.732,1.5)
        (-1.732, -1) -- (-0.866,-0.5)
        (1.732, -1) -- (0.866,-0.5)
        (-0.866,-1.5) -- (0,-1)
        (0.866,-1.5) -- (0,-1)
        (-2.599, -0.5) -- (-1.732,0)
        (0.866, -0.5) -- (0.866,1.5)
        
        ;
        \draw (-1.732,-0.375) node[anchor=center]{$\Delta_1^{(2)}$};
        \draw (-0.866,-0.875) node[anchor=center]{$\Delta_1^{(1)}$};
        \draw (0.866,-0.875) node[anchor=center]{$\Delta_2^{(1)}$};
        \draw (0.433,1) node[anchor=center]{$\Delta_1^{(1)}$};
        \draw (-0.9,1) node[anchor=center]{$\Delta_1^{(2)} = \Delta_2^{(1)}$};
        \draw (-2.2,1) node[anchor=center]{$Y_1$};
        \draw (1.5,1) node[anchor=center]{$Y_2$};
        \draw (0,-1.5) node[anchor=center]{$Y_3$};
        \draw (0,-0.5) node[anchor=center]{$ \Delta_1^{(1)}= \Delta_2^{(1)}$};
        
    \end{tikzpicture}
    \hspace{2cm}
    \begin{tikzpicture}[scale=1.2]
\draw   

        (-1.732, -1) -- (0,2)       
        (-1.732, -1) -- (1.732, -1)
        
        (1.732, -1) -- (0,2)
        
        ;
        \draw[red] (-0.433,1.25) -- (0.433,1.25);
        \draw[red] (-0.866,-1) -- (0.433,1.25);
        \draw[red] (-1.3,-0.25) -- (1.3,-0.25);
        \draw (-1.7,-0.25) node[anchor=center, color = red]{$\Delta_1^{(1)}$};
        \draw (-0.833,1.25) node[anchor=center, color = red]{$\Delta_1^{(2)}$};
        \draw (1.4,1.25) node[anchor=center, color = red]{$\Delta_1^{(2)}= \Delta_2^{(1)}$} ;
        \draw (1.8,-0.25) node[anchor=center, color = red]{$\Delta_1^{(1)}$};
        \draw (-0.18,0) node[anchor=center, color = red]{$\Delta_1^{(1)} = \Delta_2^{(1)}$};
        \draw (-0.866,-1.3) node[anchor=center, color = red]{$\Delta_2^{(1)}$};
        \filldraw[red] (-1.3,-0.25) circle (2pt) ;
        \filldraw[red] (1.3,-0.25) circle (2pt) ;
        \filldraw[red] (-0.866,-1) circle (2pt) ;
        \filldraw[black] (0,2) circle (2pt) ;
        \filldraw[black] (-1.732, -1) circle (2pt) ;
        \filldraw[black] (1.732, -1) circle (2pt) ;
        \filldraw[red] (-0.433,1.25) circle (2pt) ;
        \filldraw[red] (0.433,1.25) circle (2pt) ;
        \filldraw[red] (-0.433,-0.25) circle (2pt) ;
        \draw (0,2.3) node[anchor=center]{$Y_1$};
        \draw (-1.832, -1.3) node[anchor=center]{$Y_3$};
        \draw (1.832, -1.3) node[anchor=center]{$Y_2$};
\end{tikzpicture}
    \end{center}
    \caption{Geometric and tropical picture at $t_1 = t_2 = t_3 =0$ in $X[\{1,2\},\{3\}]$.}
    \label{mixed and pure}
\end{figure}

\medskip
\emph{$\Delta_i$-multiplicity of the components.} Again, let $t_i$ and $t_j$ be two consecutive zero entries of $(t_1,\dots,t_{n+1})$. If $i\in A$, the $\Delta_1$-multiplicity of the expanded component will be equal to the number of elements of the set $\{i,\dots, j-1\}$ which are contained in $A$. Recall that if $i\notin A$, then its $\Delta_1$-multiplicity is zero since it is pure of type 2. Similarly, if $j\in B$, then the $\Delta_2$-multiplicity of this component is given by the number of elements of the set $\{i+1,\dots, j\}$ contained in $B$.

\begin{remark}\label{remark about unbroken condition}
    The condition in Definition \ref{unbroken} is there to ensure that all modified special fibres have the expected base codimension (see Definition \ref{base codimension}). Indeed, if we allowed, for example, $i\in B\setminus A$ and $j\in A\setminus B$ with $i<j$ and the blow-ups as given above, then when $t_i = t_j=0$ and all other basis directions are non-vanishing, we would see a copy of $X_0$. Then we would need to identify such a fibre with other copies of $X_0$ in our stack $\mathfrak{X}'$, but this would imply identifying fibres of different base codimension, which cannot be done through isomorphisms on the base. To avoid this unpleasantness, we therefore fix an ordering of the elements of $A\setminus B$, $ A\cap B$ and $B\setminus A$.
\end{remark}

\begin{remark}\label{remark unbroken trop}
    Any modified special fibre in $X[A,B]$ can be expressed by adding certain edges and vertices to the triangle $\trop(X_0)$. Moreover, we notice that the sequence of blow-ups which yields this modification of $X_0$ is completely encoded in the $Y_1\cap Y_2$ edge of this triangle. Indeed, for each integral point on this edge, it suffices to note whether it has an edge of type 1 or 2 attached to it, and this information completely determines the modification.

    The unbroken condition states that, in order from $Y_1$ to $Y_2$, we should see first integral points of $Y_1\cap Y_2$ with only edges of type 2 attached, then points with both types of edges attached and, finally, on the side closest to $Y_2$, we should see points with only edges of type 1 attached to them.
\end{remark}

\subsubsection*{Viewing $X[A,B]$ as a sublocus of a larger scheme $X[A',B']$.}
Let $n'>n$, let $(A',B',n')$ be an unbroken pair, and suppose that $|A|\leq |A'|$ and $|B|\leq |B'|$, where $|A\setminus B| = i$, $|A\cap B| = j-i$ and $|B\setminus A| = n-j$. In this case, we may define a natural inclusion
\begin{equation}\label{AB embed}
    \iota \colon C[A,B]\longhookrightarrow C[A',B'],
\end{equation}
given by
\[
\begin{tikzcd}
    (t_1,\dots, t_i,t_{i+1},\dots,t_j,t_{j+1}\dots, t_{n}) \arrow[d, mapsto] \\
    (t_1,\dots, t_i,1,\dots,1,t_{i+1},\dots,t_j,1,\dots,1,t_{j+1}\dots, t_{n},1,\dots,1).
\end{tikzcd}
\]
This, in turn, determines an embedding
\[
X[A,B]\longhookrightarrow X[A',B'].
\]

\begin{proposition}
    Given any two unbroken pairs $(A_1,B_1,n_1)$ and $(A_2,B_2,n_2)$ there exists a common refinement $(A,B,n)$ such that both $X[A_1,B_1]$ and $X[A_2,B_2]$ embed into $X[A,B]$.
\end{proposition}

\begin{proof}
    We start by picking any integer $l_{A\setminus B}$ such that $|A_1\setminus B_1|,|A_2\setminus B_2|\leq l_{A\setminus B}$. We then choose any integer $l_{A\cap B}$ such that $|A_1\cap B_1|,|A_2\cap B_2|\leq l_{A\cap B}$. Finally, we take any integer $ l_{B\setminus A}$ such that $|B_1\setminus A_1|,|B_2\setminus A_2|\leq l_{B\setminus A}$. Now, recalling the above notation $[k]\coloneqq \{1, \dots, k\}$, let
    \[
    A\coloneqq [l_{A\setminus B} + l_{A\cap B}] \quad \textrm{and} \quad B\coloneqq [ l_{A\cap B} + l_{B\setminus A}].
    \]
    Then the embeddings \eqref{AB embed} on the base described above induce the desired embeddings.
    
\end{proof}

\subsection{The group action and linearisation}

\subsubsection*{The group action.}
As the constructions here are generalisations of the construction $X[n]$ in \cite{CT}, we may define a group action on them in a similar way. We recall the details here for convenience.

\medskip
There is a global action of the maximal diagonal torus $G\subset \SL(n+1)$ on $X[A,B]$, which we may describe as follows. There is an isomorphism $\GG_m^n\cong G\subset\GG_m^{n+1}$, allowing us to regard an element of $G$ as an $(n+1)$-tuple $(\sigma_1,\dots,\sigma_{n+1})$ where $\prod_i \sigma_i = 1$. This acts naturally on $\AA^{n+1}$, which induces an action on $C[n]$. The isomorphism $\GG_m^n\cong G$ is given by
\[
(\tau_1, \ldots, \tau_{n})\longrightarrow (\tau_1, \tau_1^{-1}\tau_2, \ldots, \tau_{n-1}^{-1}\tau_{n}, \tau_n^{-1}).
\]
As in \cite{CT}, we use the notation $(\tau_1, \ldots, \tau_{n})$ to describe elements of $G$ throughout this work. Now, if a $\Delta_1^{(k)}$ component exists in $X[A,B]$, we denote its local $\PP^1$ coordinates by $(x_0^{(k)}:x_1^{(k)})$. Similarly, if a $\Delta_2^{(k)}$ component exists in $X[A,B]$, we denote its local $\PP^1$ coordinates by $(y_0^{(k)}:y_1^{(k)})$. The following proposition is immediate from \cite{CT}.

\begin{proposition}\label{group proposition}
Let $(A,B,n)$ be an unbroken pair and let $l_{A}\coloneqq |A|$ and $l_{B}\coloneqq |B|$. We have the following properties.
    \begin{enumerate}
        \item There is a unique $G$-action on $X[A,B]$ such that $X[A,B]\to X\times_{\AA^1}\AA^{n+1}$ is equivariant with respect to the natural action of $G$ on $\AA^{n+1}$.
        \item In the étale local model, this action is the restriction of the action on $(X\times_{\AA^1}\AA^{n+1})\times (\PP^1)^{l_A+l_B}$, which is trivial on $X$, acts by
        \begin{align*}
            t_1 &\longmapsto \tau_1^{-1}t_1  \\
            t_k &\longmapsto \tau_k^{-1}\tau_{k-1} t_k \\
            t_{n+1} &\longmapsto \tau_{n} t_{n+1}
        \end{align*}
        on the basis directions, and by
        \begin{align*}
            (x_0^{(k)}:x_1^{(k)}) &\longmapsto (\tau_k x_0^{(k)}: x_1^{(k)})  \\
            (y_0^{(k)}:y_1^{(k)}) &\longmapsto (y_0^{(k)}: \tau_{n+1-k} y_1^{(k)})
        \end{align*}
        on the components $\Delta_1^{(k)}$ and $\Delta_2^{(k)}$ which exist in $X[A,B]$.
        \item The embeddings $X[A,B] \hookrightarrow X[A',B']$ given in \eqref{AB embed} are equivariant under the group action.
    \end{enumerate}
\end{proposition}

\subsubsection*{The $G$-linearisation.}
Similarly to \cite{CT}, we may embed $X[A,B]$ into a product of projective bundles. This is done again by defining vector bundles
\begin{align*}
    \cF_{1}^{(i)} &= \pr_1^* \cO_X(-Y_{1}) \oplus \pr_2^* \cO_{\AA^{n+1}}(-V(t_{i})) \\
    \cF_{2}^{(j)} &= \pr_1^* \cO_X(-Y_{2}) \oplus \pr_2^* \cO_{\AA^{n+1}}(-V(t_{n+2-j}))
\end{align*}
on $X\times_{\AA^1}\AA^{n+1}$ for $i\in A$ and $j\in B$, where $\pr_i$ denote the natural projections from $X\times_{\AA^1}\AA^{n+1}$ to $X$ and $\AA^{n+1}$. Then the following lemma follows directly from Lemma 3.2.1 of \cite{CT}.

\begin{lemma}
    There is an embedding
    \[
    X[A,B] \longhookrightarrow \prod_{i,j}\PP( \cF_i^{(j)}),
    \]
    where the product of projective bundles $\prod_{i,j}\PP( \cF_i^{(j)})$ is constructed as a fibre product over $X\times_{\AA^1}\AA^{n+1}$. Etale locally this corresponds to the embedding in $(X\times_{\AA^1}\AA^{n+1})\times (\PP^1)^{l_A+l_B}$.
\end{lemma}

We may then construct the following $G$-linearised ample line bundles on $X[A,B]$. 

\begin{lemma}\label{ample bundle L}
    There exists a $G$-linearised ample line bundle $\cL$ on $X[A,B]$ such that locally the lifts to this line bundle of the $G$-action on each $\PP^1$ corresponding to a $\Delta_1^{(i)}$ and on each $\PP^1$ corresponding to a $\Delta_2^{(n+2-j)}$ are given by
    \begin{align}
        \label{lift1}
        (x_0^{(i)};x_1^{(i)}) &\longmapsto (\tau_i^{a_i} x_0^{(i)} ; \tau_i^{-b_i} x_1^{(i)}) \\
        \label{lift2}
        (y_0^{(n+2-j)};y_1^{(n+2-j)}) &\longmapsto (\tau_{j+1}^{-c_j} y_0^{(n+2-j)} ; \tau_{j+1}^{d_j} y_1^{(n+2-j)}) 
    \end{align}
    for any $i\in A$ and $j\in B$ and any choice of positive integers $a_i,b_i,c_j,d_j$.
\end{lemma}

\begin{proof}
    This follows directly from Lemma 3.2.2 of \cite{CT}.
\end{proof}

As before, we will apply the relative Hilbert-Mumford criterion of \cite{HM} to define GIT stability conditions with respect to the $G$-linearised line bundles we defined. We leave this discussion for Section \ref{2nd stability}.

\section{The stack construction}\label{stack section}
Building upon the scheme constructions $X[A,B]\to C[A,B]$, we define a stack of expansions and its family, which we denote $\mathfrak{X}'\to \mathfrak{C}'$. The stacks $\mathfrak{X}$ and $\mathfrak{C}$ of \cite{CT} are substacks of $\mathfrak{X}'$ and $\mathfrak{C}'$; the latter contains more choices of expansions of $X$. Note that it does not contain all possible choices, as we have made restrictions for geometric reasons on the type of expansions we allow. For example, in order to construct a large family containing all the expansions, we chose them so that the order of the blow-ups commutes.

\subsection{The stack of expansions}\label{2nd stack}

In the following we define the stack of expansions $\mathfrak{C}'$. This will differ slightly from \cite{CT} and the construction of Li and Wu in that here we want to remember what components of the special fibre $X_0$ are blown up along which basis directions, i.e.\ the data of the sets $A$ and $B$. We will start by defining isomorphisms which when lifted to the family will effectively identify any two isomorphic fibres.

\subsubsection*{The isomorphisms.}
Let us consider $\AA^{n+1}$ with its natural torus action $\GG_m^{n}$ as before and denote by $(\AA^{n+1})_{[A,B]}$ this space enhanced by a choice of sets $A$ and $B$ such that $(A,B,n)$ is unbroken. We label elements of $(\AA^{n+1})_{[A,B]}$ by $(t_1,\dots, t_{n+1})$ as before and, for any $k$, we use again the notation $[k] \coloneqq \{1,\dots, k\}$.
Let $J \subseteq [n+1]$ and let $J^\circ = [n+1]\setminus J$ be its complement. We denote the cardinality of $J$ by $r\coloneqq |J|$ and define
\begin{align*}
    &\ind_{J}\colon [r] \longrightarrow  J \quad \textrm{and} \\
    &\ind_{J^\circ}\colon \{ r+1,\dots, n+1 \} \longrightarrow  J^\circ
\end{align*}
to be the order preserving isomorphisms.

\medskip
Remark that, although the map $\ind_J$ is defined identically to the map of the same name in Section 5 of \cite{CT}, the map $\ind_{J^\circ}$ is given a different definition here. This change is introduced so that we may keep track of which indices of the entries of $(t_1,\dots,t_{n+1})$ lie in $A$ or $B$ in the following isomorphisms.

\medskip
Let
\[
(\AA^{n+1}_{J})_{[A,B]} = \{ (t)\in (\AA^{n+1})_{[A,B]} |\ t_i = 0,\ i\in J \} \subset (\AA^{n+1})_{[A,B]}
\]
and
\[
(\AA^{n+1}_{U(J)})_{[A,B]} = \{ (t)\in (\AA^{n+1})_{[A,B]} |\ t_i \neq 0,\ i\in J^\circ \} \subset (\AA^{n+1})_{[A,B]}.
\]
Now, given $(\AA^{n+1}_{U(J)})_{[A,B]}$, we define a corresponding space $(\AA^r \times \GG_m^{n+1-r})_{[\epsilon(A),\epsilon(B)]}$ by specifying sets $\epsilon(A)$ and $\epsilon(B)$ in the following way. If $\ind_J(1)\in A\cap B$, then let $1\in \epsilon(A)\setminus \epsilon(B)$. Then, for $i>1$, if $\ind_J(i)\in A\cap B$ then $i\in \epsilon(A)\cap \epsilon(B)$. For $i\geq 1$, if $\ind_J(i)\in A\setminus B$ then $i\in \epsilon(A)\setminus \epsilon(B)$ and if $\ind_J(i)\in B\setminus A$ then $i\in \epsilon(B)\setminus \epsilon(A)$. Finally, let $i\in \epsilon(B)\setminus \epsilon(A)$ for all $i>r$. We note that $(\epsilon(A),\epsilon(B),n)$ is unbroken. We may then define the isomorphism
\[
\tau_J^{[A,B]} \colon (\AA^r \times \GG_m^{n+1-r})_{[\epsilon(A),\epsilon(B)]} \longrightarrow (\AA^{n+1}_{U(J)})_{[A,B]}
\]
given by
\[
(a_1, \ldots, a_{r},\sigma_{r+1},\dots, \sigma_{n+1}) \longrightarrow
(t_1,\dots, t_{n+1}),
\]
where 
\begin{align*}
    & t_i = a_j, \ \mathrm{ if } \ \ind_{J}(j)=i, \\
    & t_i = \sigma_j, \ \mathrm{ if } \ \ind_{J^\circ}(j)=i.
\end{align*}
Essentially, here, we have rearranged the order of the basis directions, so that those which can vanish are given by the $a_i$ entries at the front of the basis vector. These entries retain their ordering relative to each other and their partition with respect to the unbroken pair, except that the index of the first entry may go from $A\cap B$ to $\epsilon(A)\setminus \epsilon(B)$. The $\sigma_i$ entries correspond to the nonzero $t_i$ and their indices are all contained in $\epsilon(B)\setminus \epsilon(A)$.

Given two triples $(A,B,J)$ and $(A',B',J')$ such that $|J|=|J'|$ and such that $\epsilon(A) = \epsilon(A')$ and $\epsilon(B) = \epsilon(B')$, we may then define an isomorphism
\[
\tau_{J,[A,B]}^{J',[A',B']} = \tau_J^{[A,B]} \circ (\tau_{J'}^{[A',B']})^{-1} \colon (\AA^{n+1}_{U(J')})_{[A',B']} \longrightarrow (\AA^{n+1}_{U(J)})_{[A,B]}.
\]
Now, we define a second set of isomorphisms similarly to the above. Given any $(\AA^{n+1}_{U(J)})_{[A,B]}$, we may define an alternative space $(\AA^r \times \GG_m^{n+1-r})_{[\delta(A),\delta(B)]}$ by describing sets $\delta(A)$ and $\delta(B)$ in the following way. If $\ind_J(r)\in A\cap B$, then let $r\in \delta(B) \setminus \delta(A)$. For $i<r$, if $\ind_J(i)\in A\cap B$, then let $i\in \delta(A)\cap \delta(B)$. For $i\leq r$, if $\ind_J(i)\in A\setminus B$, then let $i\in \delta(A)\setminus \delta(B)$ and if  $\ind_J(i)\in B\setminus A$ let $i\in \delta(B)\setminus \delta(A)$. Again, the pair $(\delta(A),\delta(B),n)$ is unbroken. We then define an isomorphism
\[
\rho_J^{[A,B]} \colon (\AA^r \times \GG_m^{n+1-r})_{[\delta(A),\delta(B)]} \longrightarrow (\AA^{n+1}_{U(J)})_{[A,B]}
\]
given by
\[
(a_1, \ldots, a_{r},\sigma_{r+1},\dots, \sigma_{n+1}) \longrightarrow
(t_1,\dots, t_{n+1}),
\]
where 
\begin{align*}
    & t_i = a_j, \ \mathrm{ if } \ \ind_{J}(j)=i, \\
    & t_i = \sigma_j, \ \mathrm{ if } \ \ind_{J^\circ}(j)=i.
\end{align*}
As above, given two triples $(A,B,J)$ and $(A',B',J')$ such that $|J|=|J'|$ and such that $\delta(A) = \delta(A')$ and $\delta(B) = \delta(B')$, we may then specify an isomorphism
\[
\rho_{J,[A,B]}^{J',[A',B']} = \rho_J^{[A,B]} \circ (\rho_{J'}^{[A',B']})^{-1} \colon (\AA^{n+1}_{U(J')})_{[A',B']} \longrightarrow (\AA^{n+1}_{U(J)})_{[A,B]}.
\]

\subsubsection*{The stack $\mathfrak{C}'$.}
Finally, given an unbroken pair $(A,B,n)$, we define $\mathfrak{U}^{A,B}$ to be the quotient $[(\AA^{n+1})_{[A,B]}/\!\!\sim]$ by the equivalences generated by the $\GG_m^{n}$-action and the equivalences $\tau_{J,[A,B]}^{J',[A',B']}$ and $\rho_{J,[A,B]}^{J',[A',B']}$ for compatible triples $(A,B,J)$ and $(A',B',J')$. Recall from Section \ref{2nd BUs} that we had natural inclusions \eqref{AB embed}
\[
C[A,B]\longhookrightarrow C[A'',B''],
\]
for all $A''$ and $B''$ such that $(A'',B'',n'')$ is unbroken for some $n''>n$ and such that $|A|\leq |A''|$ and $|B|\leq |B''|$. These induce open immersions of stacks
\begin{align*}
    \mathfrak{U}^{A,B} \longrightarrow \mathfrak{U}^{A'',B''}.
\end{align*}
Let $\mathfrak{U}' \coloneqq \lim\limits_{\to} \mathfrak{U}^{A,B}$ be the direct limit over all unbroken pairs $(A,B,n)$ and let $\mathfrak{C}' \coloneqq C\times_{\AA^1} \mathfrak{U}'$.

\medskip
We now make a few remarks concerning the isomorphisms described in this section. As mentioned above, these isomorphisms are effectively a reordering of the basis directions, where the $t_i$'s whose indices lie in $J$ preserve their order relative to each other and the corresponding indices preserve their $A,B$ labelling, except in certain cases for the first and last index in $J$. As we will see in the proof of Proposition \ref{isom of fibres induced by isoms of base}, these are exactly the isomorphisms on the base which correspond to isomorphisms of the corresponding fibres above these elements.

\medskip
Let $(\AA^{n+1}_{U(J)})_{[A,B]}$ be such that $\ind_J(1)\in A\cap B$. Then it is isomorphic by the isomorphism $\tau_J^{[A,B]}$ to the space $(\AA^r \times \GG_m^{n+1-r})_{[\epsilon(A),\epsilon(B)]}$ defined above, where $1\notin \epsilon(B)$. We note that the isomorphisms $\tau_{J,[A,B]}^{J',[A',B']}$ identify $(\AA^{n+1}_{U(J)})_{[A,B]}$ with spaces $(\AA^{n+1}_{U(J')})_{[A',B']}$ where either $\ind_{J'}(1)\in A'\cap B'$ or $\ind_{J'}(1)\in A'\setminus B'$, as long as $|J|=|J'|$ and the equalities $\epsilon(A)=\epsilon(A')$ and $\epsilon(B)= \epsilon(B')$ hold.

\subsection{The family $\mathfrak{X}'$ over $\mathfrak{C}'$.}\label{family over stack of expansions}

Let $X[A,B]\to C[A,B]$ be a construction as described in Section \ref{2nd BUs} and recall that $\pi\colon X[A,B] \to X$ is the projection to the original family. Let $n'>n$, let $(A',B',n')$ be unbroken, and let $|A|\leq |A'|$ and $|B|\leq |B'|$. In such a case, we defined a natural inclusion
\[
\iota \colon C[A,B] \longhookrightarrow C[A',B']
\]
in \eqref{AB embed}. Then the induced family $(\iota^* X[A',B'],\iota^* \pi)$ is isomorphic to $(X[A,B],\pi)$ over $C[A,B]$. The following proposition shows that the equivalences defined on the stack $\mathfrak{U}^{A,B}$ lift to $C$-isomorphisms of fibres. Let $X[A,B]_{U(J)}$ and $X[A,B]_{J}$ be the restrictions of $X[A,B]$ to $(\AA^{n+1}_{U(J)})_{[A,B]}$ and $(\AA^{n+1}_{J})_{[A,B]}$ respectively.

\begin{proposition}\label{isom of fibres induced by isoms of base}
Let $(A,B,n)$ and $(A',B',n)$ define two unbroken pairs. The schemes $X[A,B]_{U(J)}$ and $X[A',B']_{U(J')}$ are isomorphic if and only if $(\AA^{n+1}_{U(J)})_{[A,B]}$  and $(\AA^{n+1}_{U(J')})_{[A',B']}$ are related by the isomorphisms $\tau_{J,[A,B]}^{J',[A',B']}$ and $\rho_{J,[A,B]}^{J',[A',B']}$.
\end{proposition}

\begin{proof}
    It is clear that $X[A,B]_{U(J)}$ and $X[A',B']_{U(J')}$ are isomorphic on $\pi^{-1}_{[A,B]} (X^\circ) \cong \pi^{-1}_{[A',B']} (X^\circ)$, i.e.\ on the locus where all basis directions are nonzero. We must show that for each modified special fibre in $X[A,B]_{U(J)}$ there is an isomorphic fibre in $X[A',B']_{U(J')}$ if and only if the corresponding bases are related by $\tau_{J,[A,B]}^{J',[A',B']}$ and $\rho_{J,[A,B]}^{J',[A',B']}$. Then the proposition follows from the fact that if there is a bijection between the isomorphism classes of fibres of $X[A,B]$ and those of $X[A',B']$, the two spaces must be isomorphic as schemes, as they must both be given by the same sequence of blow-ups on $X\times_{\AA^1}\AA^{n+1}$. As the blow-ups commute, we may discover this sequence just by looking at the fibres where all entries of $(t_1,\dots, t_{n+1})$ are zero (in fact, by construction, it is enough to show that these two fibres of maximal base codimension are isomorphic to show that the schemes $X[A,B]$ and $X[A',B']$ are isomorphic).

    \medskip
    First, assume there is a sequence of $\tau_{J,[A,B]}^{J',[A',B']}$ and $\rho_{J,[A,B]}^{J',[A',B']}$ isomorphisms such that $(\AA^{n+1}_{U(J)})_{[A,B]} \cong (\AA^{n+1}_{U(J')})_{[A',B']}$. Since $|J|=|J'|=r$, this implies that, for $1<i<r$, the element $\ind_J(i)$ belongs to $A\setminus B$ if and only if $\ind_{J'}(i)$ belongs to $A'\setminus B'$. The same goes for the sets $A\cap B$ with $A'\cap B'$ and $B\setminus A$ with $B'\setminus A'$.

    Next, we note that the placement of the nonzero entries in $(t_1,\dots, t_{n+1})$ influences the $\Delta$-multiplicity of the components but has no effect on how many components are expanded out in the fibre above this point or whether these expanded components are of type $\Delta_1$ or type $\Delta_2$. Moreover, changing the first zero entry of $(t_1,\dots,t_{n+1})$ from an element of $A\setminus B$ to an element of $A\cap B$ or vice versa also does not affect which components are expanded out in this fibre. The same is true of changing the last zero entry of $(t_1,\dots,t_{n+1})$ from $B\setminus A$ to $A\cap B$ or vice versa. This can be easily seen by studying the equations of the blow-ups. It therefore must follow that a sequence of isomorphisms $\tau_{J,[A,B]}^{J',[A',B']}$ and $\rho_{J,[A,B]}^{J',[A',B']}$, as they only modify the base in ways that preserve the fibres, induces an isomorphism $X[A,B]_{U(J)}\cong X[A',B']_{U(J')}$.

    \medskip
    Conversely, let us assume that all fibres in $X[A,B]_{U(J)}$ are isomorphic to fibres in $X[A',B']_{U(J')}$. This means that for every fibre of $X[A,B]_{U(J)}$ over a point $(t_1,\dots,t_{n+1})$ of the base $(\AA^{n+1}_{U(J)})_{[A,B]}$, there exists a fibre of $X[A',B']_{U(J')}$ over a point $(t'_1,\dots,t'_{n+1})$ of the base $(\AA^{n+1}_{U(J')})_{[A',B']}$ such that the exact same number of components are expanded out in the same loci of both fibres and these components must have the same type. Clearly, there must be the same amount of zero entries in $(t_1,\dots,t_{n+1})$ and $(t'_1,\dots,t'_{n+1})$; call this number $k$. Moreover, for $k>1$, by studying the equations of the blow-ups we can see that for the above statement to be true, the first $k-1$ zeros of $(t_1,\dots,t_{n+1})$ must have indices belonging to $A$ if and only if the first $k-1$ zeros of $(t'_1,\dots,t'_{n+1})$ have indices belonging to $A'$; and the last $k-1$ zeros of $(t_1,\dots,t_{n+1})$ must have indices belonging to $B$ if and only if the last $k-1$ zeros of $(t'_1,\dots,t'_{n+1})$ have indices belonging to $B'$.

    If the index of the first zero in $(t_1,\dots,t_{n+1})$ is not contained in $A$, then it must be contained in $B\setminus A$ and all expanded components in the fibre above this point are of type $\Delta_2$. This can only be isomorphic to the fibre above the point $(t'_1,\dots,t'_{n+1})$ if the indices of all zero entries in $(t'_1,\dots,t'_{n+1})$ are contained in $B'\setminus A'$.  If the index of the first zero in $(t_1,\dots,t_{n+1})$ is contained in $A$ and the indices of all subsequent zeros are contained in $B$, however, then whether or not this first index is contained in $B$ has no effect on the components expanded out in this fibre. In conclusion, if the fibres above points $(t_1,\dots,t_{n+1})$ and $(t'_1,\dots,t'_{n+1})$ are isomorphic, then the index of the first zero entry of $(t_1,\dots,t_{n+1})$ is in $B\setminus A$ if and only if the index of the first zero entry of $(t'_1,\dots,t'_{n+1})$ is in $B'\setminus A'$; and if the index of the first zero entry of $(t_1,\dots,t_{n+1})$ is in $A\cap B$ then the index of the first zero entry of $(t'_1,\dots,t'_{n+1})$ is in either of $A'\setminus B'$ or $A'\cap B'$. A similar reasoning holds with respect to the index of the last zero entry being contained in $B\setminus A$ or $A\cap B$. Indeed, if this index is contained in $A\cap B$, then $\ind_{J'}(k)$ may be contained in $A'\cap B'$ or $B'\setminus A'$. But these are exactly the elements of the base which are related by the isomorphisms $\tau_{J,[A,B]}^{J',[A',B']}$ and $\rho_{J,[A,B]}^{J',[A',B']}$.
\end{proof}

We define $\mathfrak{X}^{A,B}$ to be the quotient $[X[A,B]/\!\!\sim]$ by the equivalences generated by the $\GG_m^{n}$-action and equivalences lifted from $\mathfrak{U}^{A,B}$. There are natural immersions of stacks
\begin{align*}
    \mathfrak{X}^{A,B} \longhookrightarrow \mathfrak{X}^{A',B'},
\end{align*}
induced by the immersions $\mathfrak{U}^{A,B} \hookrightarrow \mathfrak{U}^{A',B'}$. Finally, we define $\mathfrak{X} = \lim\limits_{\to} \mathfrak{X}^{A,B}$ to be the direct limit over all unbroken pairs $(A,B,n)$. It is an Artin stack (here we mean as opposed to a Deligne-Mumford stack, i.e.\ it has infinite automorphisms).

\subsection{Stability conditions}\label{2nd stability}

We now discuss stability conditions on the stack $\mathfrak{X}'$. We generalise the two perspectives given in \cite{CT} to this case, namely that coming from the GIT stability conditions arising from the $G$-linearised line bundles we constructed in Section \ref{2nd BUs}, and that which relates to the stability conditions presented by Li and Wu in \cite{LW} and Maulik and Ranganathan in \cite{MR}.

\subsubsection*{GIT stability on the relative Hilbert scheme of points.}
Let $(A,B,n)$ be an unbroken pair and denote by
\[
H^m_{[A,B]} \coloneqq \Hilb^m(X[A,B]/C[A,B])
\]
the relative Hilbert scheme of $m$ points. Recall that in Lemma \ref{ample bundle L}, we gave possible choices for a $G$-linearised ample line bundle $\cL$ on $X[A,B]$. As in \cite{CT}, this will induce a $G$-linearised ample line bundle on $H^m_{[A,B]}$ in the following way. Let
\[
Z^m_{[A,B]}\subset H^m_{[A,B]}\times_{C[A,B]} X[A,B]
\]
be the universal family, with first and second projections $p$ and $q$. The line bundle
\[
\cM_l\coloneqq \det p_*(q^*\cL^{\otimes l}|_{Z^m_{[A,B]}})
\]
is relatively ample when $l\gg0$ and is $G$-linearised, by the same argument as in \cite{GHH}. The choices of line bundle $\cM_l$ on $H^m_{[A,B]}$ arising in this way paired with choices of lifts of the group action to $\cM_l$ are enough to allow us to construct the desired stable locus. The following results follow directly from Theorem 4.4.3 and Lemma 4.4.1 of \cite{CT}.

\begin{theorem}\label{stability theorem}
    Let $(A,B,n)$ be an unbroken pair and let $Z$ be a length $m$ zero-dimensional subscheme in a fibre of $X[A,B]$. Then there exists a GIT stability condition on $H^m_{[A,B]}$ which makes $Z$ stable (as opposed to strictly semistable) if and only if there is at least one point of the support of $Z$ in the union $(\Delta_1^{(k)})^\circ \cup (\Delta_2^{(n+1-k)})^\circ$ for every $k$.
\end{theorem}

\begin{corollary}
    Theorem \ref{stability theorem} still holds if we restrict the possible GIT stability conditions to just those arising from choices of line bundles $\cM_l$ as defined above.
\end{corollary}

\subsubsection*{Li-Wu and modified GIT stability.} We extend the notions of stability set up in \cite{CT}, namely LW and SWS stability, first on schemes, then generalise to stacks. In section \ref{choices of stacks}, we will introduce an additional stability condition needed to obtain proper algebraic stacks. Throughout this section we assume that $(A,B,n)$ is an unbroken pair.

\begin{definition}
    Let $Z$ be a length $m$ zero-dimensional subscheme in a fibre of $X[A,B]$. We say that $Z$ is \emph{Li-Wu (LW) stable} if $Z$ is smoothly supported and has finite automorphisms. We denote the LW stable locus by $H^m_{[A,B],\LW}$.
\end{definition}

\begin{remark}
    In the case of Hilbert schemes of points, this happens to correspond to the notion of stability defined in \cite{LW}. If we wanted to generalise these results to Hilbert schemes with non-constant polynomials, however, we would need to take the original definition of Li-Wu stability which can be found in \cite{LW} or \cite{CT}.
\end{remark}

We also extend the definition of SWS stability to this setting in the obvious way as follows.

\begin{definition}
    Let $Z$ be a length $m$ zero-dimensional subscheme in a fibre of $X[A,B]$. We say that $Z$ is \textit{weakly strictly stable} if there exists any $G$-linearised ample line bundle $\cM'$ on $H^m_{[A,B]}$ with respect to which $Z$ is GIT stable. The term \emph{strictly} is used here to emphasize that we exclude strictly semistable points. If $Z$ is also supported in the smooth locus of the fibre, then we say it is \textit{smoothly supported weakly strictly stable} (abbreviated to \emph{SWS stable} as before) and we denote the SWS stable locus by $H^m_{[A,B],\SWS}$.
\end{definition}

We then have the inclusion
\[
H^m_{[A,B],\SWS} \subset H^m_{[A,B],\LW},
\]
since, if points are GIT stable, they must have finite stabilisers. We now extend our definition of SWS stability to the stack $\mathfrak{X}'$.

\medskip
Given any $C$-scheme $S$, an object of $\mathfrak{X}'(S)$ is a pullback family $\xi^* X[A,B]$ for a morphism
\[
\xi \colon S \to C[A,B].
\]

\begin{definition}
    Let $\mathcal{X}\in \mathfrak{X}'(S)$ and let $Z$ be family of length $m$ zero-dimensional subschemes in $\cX$. The pair $(Z,\mathcal{X})$ is said to be \textit{SWS stable} if and only if $\cX \coloneqq \xi^* X[A,B]$ for some morphism $\xi \colon S \to C[A,B]$, where $Z$ is supported in the smooth locus of $\cX$ and there exists some  $G$-linearised ample line bundle on $H^m_{[A,B]}$ which makes $\xi_* Z$ be GIT stable.
\end{definition}

LW stability can be extended to the stack setting similarly. It is important to note that an object $\mathcal{X}\in \mathfrak{X}'(S)$ is an equivalence class under the action of $G$ and the isomorphisms set up in Section \ref{2nd stack} which identify isomorphic fibres. In particular, we highlight that the above definition means that the entire equivalence class given by a pair $(Z,\cX)$ is SWS stable if there is any representative of the equivalence class of $\cX$ for which this pair is SWS stable.

\begin{proposition}\label{SWS stab 2}
    There exists a pair $(Z,\cX)$ for some $\cX\in\mathfrak{X}'(k)$ (where $k$ is an algebraically closed field of characteristic zero as before) which is LW stable but not SWS stable. This means there is no representative of the equivalence class of $\cX$ for which this pair is SWS stable.
\end{proposition}

\begin{proof}
    This can be seen easily by describing an example where this happens. Take $\cX$ to be a fibre of base codimension 2 where exactly one $\Delta$-component is expanded out and it is pure of type 1. Then let $Z$ be a length $m$ zero dimensional subscheme with at least part of its support lying in the interior of this expanded $\Delta_1$ component, but no point of its support lying in $Y_1$. This will be LW stable but it will never be SWS stable because no matter what representative of the equivalence class of $\cX$ we choose, this representative will always have at least one $\Delta_1$- and one $\Delta_2$-component equal to $Y_1$. By Theorem \ref{stability theorem}, the pair cannot be SWS stable.
\end{proof}

\begin{remark}
    If we restrict ourselves to the substack $\mathfrak{X}$, the LW and SWS stability conditions are the same up to the equivalences of the stack. In that context, if there is at least one point of the support of $Z$ in the union $(\Delta_1^{(i)})^\circ \cup (\Delta_2^{(n+1-i)})^\circ$ for every $i$ for which these components are expanded out in $\cX$, then $(Z,\cX)$ is SWS stable. This is because in $\mathfrak{X}$ we have restricted the type of modified fibre we can have, so that examples like the one given in the proof of Proposition \ref{SWS stab 2} are not included. We may easily see that in this restricted setting, given any such pair $(Z,\cX)$, there is an equivalent pair under the isomorphisms of the stack $\mathfrak{X}$ satisfying the condition that at least one point of the support of $Z$ lies in the union $(\Delta_1^{(i)})^\circ \cup (\Delta_2^{(n+1-i)})^\circ$ for every $i$. Such a pair is therefore SWS stable by Theorem \ref{stability theorem}.
\end{remark}

\subsubsection*{Stacks of stable objects.}
Let us denote by $\mathfrak{N}^m_{\SWS}$ and $\mathfrak{N}^m_{\LW}$ the stacks of SWS and LW stable length $m$ zero-dimensional subschemes in $\mathfrak{X}'$. Let $S$ be a $C$-scheme. An object of $\mathfrak{N}^n_{\SWS}(S)$ is defined to be a pair $(Z,\cX)$, where $\cX\in \mathfrak{X}'(S)$ and $Z$ is an $S$-flat SWS stable family in $\cX$. Similarly, an object of $\mathfrak{N}^n_{\LW}(S)$ is a pair $(Z,\cX)$, where $\cX\in \mathfrak{X}'(S)$ and $Z$ is an $S$-flat LW stable family in $\cX$.

\begin{remark}
    Unlike the stacks $\mathfrak{M}^m_{\SWS}$ and $\mathfrak{M}^m_{\LW}$ of \cite{CT}, the stacks $\mathfrak{N}^m_{\SWS}$ and $\mathfrak{N}^m_{\LW}$ are universally closed but not separated, meaning that stable families of length $m$ zero-dimensional subschemes in these stacks may have more than one limit. As a consequence, Li-Wu stability and Maulik-Ranganathan stability do not coincide on these stacks. Indeed, it does not yet make sense here to speak of Maulik-Ranganathan stability on the stack $\mathfrak{X}'$ as we have not defined a notion of tube component and a choice of stable limit among all possible limits for a given a family $(Z,\cX)$. As we will see, in order to construct separated stacks we will need to select one of two options. The first is to identify all choices of representatives for a limit, which will yield a stack which is no longer algebraic. This parallels work of Kennedy-Hunt \cite{K-H}. The second is to make a choice of substack which picks out exactly one representative for each limit and these substacks will be isomorphic to the underlying stacks produced by the methods of Maulik and Ranganathan in \cite{MR} (the word underlying here is used to signify that we do not speak of logarithmic structures).
\end{remark}

\subsection{Universal closure}

We will now show that the stacks $\mathfrak{N}^m_{\SWS}$ and $\mathfrak{N}^m_{\LW}$ are universally closed. Part of this proof follows easily from the fact that $\mathfrak{X}$ is a substack of $\mathfrak{X}'$ and universal closure results of \cite{CT}; we will, however, give a more detailed proof of this statement here, as this will allow us to illustrate the fact that neither stack is separated, and where these different choices of limit representatives for a family of length $m$ zero-dimensional subschemes occur. We start by recalling the following definition from \cite{CT}.

\begin{definition}\label{extension}
    Let $S\coloneqq \Spec R \to C$, where $R$ is some discrete valuation ring and let $\eta$ denote the generic point of $S$. Now, let $(Z,\cX)$ be a pair where $\cX\in \mathfrak{X}(S)$ and $Z$ is an $S$-flat family of length $m$ zero-dimensional subschemes in $\cX$. Let $S'\to S$ be some finite base change and denote the generic and closed points of $S'$ by $\eta'$ and $\eta'_0$ respectively. We say that a pair $(Z'_{\eta_0'},\cX'_{\eta_0'})$ is an \textit{extension} of $(Z_\eta,\cX_\eta)$ if there exists such a base change and $(Z'_{\eta_0'},\cX'_{\eta_0'})$ is the restriction to $\eta'_0$ of some $S'$-flat family $(Z',\cX')$ with $\cX'\in\mathfrak{X}'(S')$ such that $Z_\eta \times_{\eta} \eta' \cong Z'_{\eta'}$ and $\cX_\eta \times_{\eta} \eta' \cong \cX'_{\eta'}$.
\end{definition}

\begin{proposition}\label{univ closed 2}
The stack $\mathfrak{N}^m_{\SWS}$ is universally closed.
\end{proposition}

\begin{proof}
Let $S\coloneqq \Spec R \to C$, where $R$ is some discrete valuation ring with uniformising parameter $w$ and quotient field $k$. We denote by $\eta$ and $\eta_0$ the generic and closed points of $S$ respectively. Let $(Z, \cX)$ be an $S$-flat family of length $m$ zero-dimensional subschemes such that $\cX\in \mathfrak{X}'(S)$ and $(Z_\eta, \cX_\eta)\in \mathfrak{N}^m_{\SWS}(\eta)$. As in the proof of Proposition 6.1.3 of \cite{CT}, we show that there exists a finite base change $S'\coloneqq \Spec R'\to S$, for some discrete valuation ring $R'$ and a pair $(Z', \cX')\in \mathfrak{N}^m_{\SWS}(S')$ satisfying the following condition. We denote by $\eta'$ and $\eta_0'$ the generic and closed points of $S'$ respectively. Then $S'$ and $(Z', \cX')$ are chosen such that we have an equivalence $\cX'_{\eta'} \cong \cX_\eta \times_\eta \eta'$ which induces an equivalence $Z'_{\eta'} \cong Z_\eta \times_\eta \eta'$.

\medskip
We start by repeating the following steps of Proposition 6.1.3 of \cite{CT}:
\begin{itemize}
    \item We choose the appropriate base change $S'\coloneqq \Spec R'\to S$ for some discrete valuation ring $R'$, which corresponds to placing the triangle $\trop(X_0)$ at a suitable height within the cone $\trop(X)$ so that the intersection vertices $\trop(Z)\cap \trop(X_0)$ are integral (see proof of Proposition 6.1.3 of \cite{CT} for more details).
    \item inductively construct an element
        \begin{equation}\label{base expression}
        (t_1,\dots, t_{n+1}) = (f_1u^{g_1},\dots, f_{n+1}u^{g_{n+1}}) \in \AA^{n+1},
        \end{equation}
        where $u$ is the uniformising parameter of $R'$ and $f_i\in (R')^\times$, i.e.\ they are units in $R'$.
        This, paired with a choice of unbroken pair $(A,B)$ for $n+1$, will determine a morphism $\xi\colon S'\to C[A,B]$ such that the pullback $\xi^*X[A,B]$ defines the required family $\cX'$.
\end{itemize}

The equation \eqref{base expression} does not on its own determine an element of $\mathfrak{X}'$ (unlike $\mathfrak{X}$). In order to construct a stable extension for $(Z,\cX)$, we still need to make a suitable choice of sets $A$ and $B$. As Proposition 6.1.3 of \cite{CT} shows, the maximal choices $A=\{1,\dots, n\}$ and $B=\{2,\dots, n+1\}$ would work, but these are not the only possible choices, as in this construction we are not forced to blow up both $x$ and $y$ along every basis directions. We will see that for this reason the choice of stable extension is not necessarily unique.

 \medskip
 Given different families of subschemes $Z$, we will describe the possible choices of $A$ and $B$ which give rise to stable extensions $(Z'_{\eta_0'},\cX'_{\eta_0'})$. Recall, from the proof of Proposition 6.1.3 in \cite{CT}, that the subscheme $Z'$ is a union of irreducible components $Z'_i$ and that $u$ was chosen such that each $Z'_i$ can be written locally in terms of its $x,y$ and $z$ coordinates as
\begin{equation}\label{xyz expression}
    \{ (c_{i,1}u^{e_{i,1}}, c_{i,2}u^{e_{i,2}}, c_{i,3}u^{e_{i,3}})\},
\end{equation}
for some $e_{i,j}\in \ZZ$ and $c_{i,j}\in (R')^\times$. As $Z'$ is a flat family given by the above expression, it must satisfy the equations
\begin{align}
        x &= c_{i,1}u^{e_{i,1}}, \nonumber \\
        y &= c_{i,2}u^{e_{i,2}}, \label{subscheme equations} \\
        z &= c_{i,3}u^{e_{i,3}} \nonumber
\end{align}
also over the closed point. Part of our stability conditions is to require subschemes to be smoothly supported, so, if more than one element of the set $\{ e_{i,1},e_{i,2},e_{i,3} \}$ is nonzero, then $Z'_{\eta'_0}$ must be supported in a component blown up along the vanishing of both sides of the relevant equations \eqref{subscheme equations}. This imposes some restrictions on the $A$ and $B$ we can choose given a certain $Z$, but as we will see there may still be choices to make.

\medskip
First, we discuss limits where the choice is unique. If $e_{i,1}$ and $e_{i,3}$ are nonzero but $e_{i,2}=0$, then the support of $(Z'_i)_{\eta'_0}$ lies in the $\pi^*(Y_1\cap Y_3)^\circ$ locus. In our construction, there is only one way of expanding a $\Delta$-component in this locus which will contain the support of $(Z'_i)_{\eta'_0}$ in its interior, namely that which in the localisation is given by blowing up the ideal $\langle x,u^{e_{i,1}} \rangle$. There is exactly one $j$ such that $t_1\cdots t_j = cu^{e_{i,1}}$, for some $c\in (R')^\times$, where $t_k$ are as in the expression \eqref{base expression}. It follows that $j\in A$.

\medskip
Similarly, if  $e_{i,2}$ and $e_{i,3}$ are nonzero but $e_{i,1}=0$, then the support of $(Z'_i)_{\eta'_0}$ lies in the $\pi^*(Y_2\cap Y_3)^\circ$ locus. Again, there is only one way of expanding a $\Delta$-component in this locus such that the flat limit of $(Z'_i)_{\eta'}$ is contained in the interior of this $\Delta$-component. In the localisation, this is given by the blow-up of the ideal $\langle y,u^{e_{i,2}} \rangle$. Now, as in the previous example, there is exactly one $j$ such that $t_{j}\cdots t_{n+1} = cu^{e_{i,2}}$, for some $c\in (R')^\times$, where $t_k$ are as in the expression \eqref{base expression} (the values of $c$ and $j$ may, of course, differ from those in the first example). It follows that $j\in B$.

\medskip
Finally, if $e_{i,1}$, $e_{i,2}$ and $e_{i,3}$ are all nonzero then the support of $(Z'_i)_{\eta'_0}$ lies in the $\pi^*(Y_1\cap Y_2\cap Y_3)$ locus which is pulled back from a codimension 2 locus in $X_0$. This implies that in order for $(Z'_i)_{\eta'_0}$ to be smoothly supported, it must be contained in a $\PP^1\times \PP^1$ bubble, given by blowing up both ideals $\langle x,u^{e_{i,1}} \rangle$ and $\langle y,u^{e_{i,2}} \rangle$. If $j$ is such that $t_1\cdots t_j = cu^{e_{i,1}}$ and $k$ is such that $t_{k}\cdots t_{n+1} = cu^{e_{i,2}}$, then we must include both $j\in A$ and $k\in B$.

\medskip
Now, on the other hand, if $e_{i,1}$ and $e_{i,2}$ are nonzero but $e_{i,3}=0$, then the support of $(Z'_i)_{\eta'_0}$ lies in the $\pi^*(Y_1\cap Y_2)^\circ$ locus and there are several possible ways to expand out a component in this locus which would contain $(Z'_i)_{\eta'_0}$ in its interior. Indeed, $(Z'_i)_{\eta'_0}$ may lie in a $\Delta_1$- or $\Delta_2$-component of pure type or in a component of mixed type $\Delta_1= \Delta_2$. In other words, if $j$ is such that $t_1\cdots t_j = cu^{e_{i,1}}$, for some $c\in (R')^\times$ and therefore $t_{j+1}\cdots t_{n+1} = du^{e_{i,2}}$, for some $d\in (R')^\times$, then we may pick any of the following three options
\begin{align*}
    &j\in A \quad \textrm{and} \quad j+1\notin B, \\
    &j\notin A \quad \textrm{and} \quad j+1\in B, \\
    &j\in A \quad \textrm{and} \quad j+1\in B.
\end{align*}
Of course, whether or not $j$ and $j+1$ are contained in $A$ and $B$ will depend also on the other $Z_i$ and is constrained by the need for $(A,B,n)$ to be unbroken, so for certain $Z$, some of these choices are removed.

\medskip 
By making compatible choices as above, we may choose an extension $(Z'_{\eta_0},\cX'_{\eta_0})$ such that at least one point of the support of $Z'_{\eta_0}$ is in the union $(\Delta_1^{(i)})^\circ \cup (\Delta_2^{(n+1-i)})^\circ$ for every $i$ for which either of these components exists in $\cX'_{\eta_0}$. By Theorem \ref{stability theorem}, this limit is SWS stable.

\medskip
 This shows that if $(Z_\eta, \cX_\eta)$ is pulled back from a fibre above a point $(t_1,\dots, t_{n+1})$ in some $C[A,B]$ whose entries are all invertible, then $(Z_\eta, \cX_\eta)$ has an SWS stable extension. See Corollary \ref{corollary limits of special elements} for a proof that there exists an extension if $\cX_\eta$ is a modified special fibre.
\end{proof}

\begin{corollary}\label{univ closed 3}
    The stack $\mathfrak{N}^m_{\LW}$ is universally closed.
\end{corollary}

\begin{proof}
    As SWS stable pairs must be LW stable, this follows immediately.
\end{proof}

\begin{remark}
    The proof of Proposition \ref{univ closed 2} shows how the stacks $\mathfrak{N}^m_{\SWS}$ and $\mathfrak{N}^m_{\LW}$ fail to be separated by highlighting several possible choices of limit representatives. We should note also that it is possible for such a pair $(Z'_{\eta_0},\cX'_{\eta_0})$ to have finite automorphisms while having expanded $\Delta$-components in $\cX'_{\eta_0}$ which contain no point of the support of $Z'_{\eta_0}$. This may happen for example if, in the expression \eqref{xyz expression}, both $e_{i,1}$ and $e_{i,3}$ are nonzero but $e_{i,2}=0$. Then, in order for $Z'_{\eta_0}$ to be smoothly supported in $\cX'_{\eta_0}$, we must let $j\in A$ (using the notation of the proof) so as to expand out the component $\Delta_1^{(j)}$; but we may also let $j+1\in B$ so as to have the component $\Delta_2^{(n+1-j)}$ expanded out in $\cX'_{\eta_0}$, even if the latter contains no point of the support. As the same $\GG_m$ acts on $\Delta_1^{(j)}$ and $\Delta_2^{(n+1-j)}$, this will not add any automorphisms. In fact, this is exactly the type of blow-up we made in \cite{CT}.
\end{remark}

\section{Constructing separated stacks.}\label{separated stack section}
We are now in a position to construct separated stacks. As mentioned earlier, there are two ways of doing this. Firstly, one can avoid making any choices and instead identify all possible limits of a given family in the stacks  $\mathfrak{N}^m_{\SWS}$ and $\mathfrak{N}^m_{\LW}$. We denote these by $\Bar{\mathfrak{N}^m_{\SWS}}$ and $\Bar{\mathfrak{N}^m_{\LW}}$ respectively. This procedure is in a sense more general, as it avoids making choices. As we will see, this comes at the cost of these stacks no longer being algebraic. The second way to obtain separated stacks is to make a systematic choice of representative for a given limit. This will be done by adding an additional stability condition.

\subsection{Non-algebraic proper stacks}\label{identifying limits}

\subsubsection*{Non-unique associated pairs.}
We recall the notion of \emph{associated pairs} from \cite{CT}.

\begin{definition}
    Let $(Z, \cX)\in \mathfrak{N}^m_{\SWS}(k)$ (or $\mathfrak{N}^m_{\LW}(k)$), where $k$ is an algebraically closed field as before. We will say that a pair $(Z, \cX)$ is an \textit{associated pair} for a configuration of vertices in $\trop(X_0)$ (or rays through these vertices in $\trop(X)$) if there is a 1-to-1 correspondence between these vertices and the non-empty bubbles in $(Z, \cX)$ given as follows: to each such vertex $(e_{i,1}, e_{i,2}, e_{i,3})\in \trop(X_0)$ corresponds a non-empty bubble blown-up along vanishing of both sides of the equations \eqref{subscheme equations}.
\end{definition}

The problem of non-separatedness comes from there being more than one associated pair in $\mathfrak{N}^m_{\SWS}$ and $\mathfrak{N}^m_{\LW}$ for a given configuration of vertices in $\trop(X_0)$. For example, in Figure \ref{assoc pairs}, we see that there are three possible stable fibre-subscheme pairs associated to the tropical picture on the left. In other words, there are several possible blow-ups of $X_0$ which create a bubble corresponding to the vertex $\trop(Z)$ in the picture on the left. Moreover, there are exactly three such blow-ups for which this bubble being the only component containing a point of the support of $Z$ creates a stable pair. The red point in each of the three geometric pictures on the right corresponds to the support of the subscheme $Z$.

\begin{figure} 
    \begin{center}  

    \begin{tikzpicture}[scale=0.7]
        \draw   

        (-1.732, -1) -- (0,2)       
        (-1.732, -1) -- (1.732, -1)
        
        (1.732, -1) -- (0,2)
        ;
        \draw (2.4,0.5) node[anchor=center, color = red]{$\trop(Z)$};
        \filldraw[red] (0.866,0.5) circle (2pt) ;
        \filldraw[black] (0,2) circle (2pt) ;
        \filldraw[black] (-1.732, -1) circle (2pt) ;
        \filldraw[black] (1.732, -1) circle (2pt) ;
        \draw (0,2.3) node[anchor=center]{$Y_1$};
        \draw (-1.832, -1.3) node[anchor=center]{$Y_3$};
        \draw (1.832, -1.3) node[anchor=center]{$Y_2$};
    \end{tikzpicture}
    \hspace{0.3cm}    
    \begin{tikzpicture}[scale=0.7]
        \draw   (-0.866,0.5) -- (0,0)
        (0,0) -- (0,2)       
        (-0.866,0.5) -- (-0.866,2)
        (-1.732, -1) -- (0,0)
        (-2.165,-0.249) -- (-0.866,0.5)
        
        (1.732, -1) -- (0,0);
        \filldraw[red] (-0.433,1) circle (2pt) ;
        \draw (-1.7,1) node[anchor=center]{$Y_1$};
        \draw (1,1) node[anchor=center]{$Y_2$};
        \draw (0,-0.8) node[anchor=center]{$Y_3$};        
    \end{tikzpicture}
    \hspace{0.3cm}
     \begin{tikzpicture}[scale=0.7]
        \draw   (0.866,0.5) -- (0,0)
        (0,0) -- (0,2)       
        (0.866,0.5) -- (0.866,2)
        (-1.732, -1) -- (0,0)
        (2.165,-0.249) -- (0.866,0.5)
        
        (1.732, -1) -- (0,0);
        \filldraw[red] (0.433,1) circle (2pt) ;
        \draw (-1,1) node[anchor=center]{$Y_1$};
        \draw (1.7,1) node[anchor=center]{$Y_2$};
        \draw (0,-0.8) node[anchor=center]{$Y_3$};
        
    \end{tikzpicture}
    \hspace{0.3cm}
    \begin{tikzpicture}[scale=0.7]
        \draw   (0.866,0.5) -- (0,0) 
        (-0.866,0.5) -- (0,0) 
        (0.866,0.5) -- (0.866,2)
        (-0.866,0.5) -- (-0.866,2)
        (-1.732, -1) -- (0,0)
        (2.165,-0.249) -- (0.866,0.5)
        (-2.165,-0.249) -- (-0.866,0.5)
        (1.732, -1) -- (0,0);
        \filldraw[red] (0,1) circle (2pt) ;
        \draw (-1.7,1) node[anchor=center]{$Y_1$};
        \draw (1.7,1) node[anchor=center]{$Y_2$};
        \draw (0,-0.8) node[anchor=center]{$Y_3$};
        
    \end{tikzpicture}
    
    \end{center}
    \caption{Three different stable associated pairs for the same tropical picture.}
    \label{assoc pairs}
\end{figure}
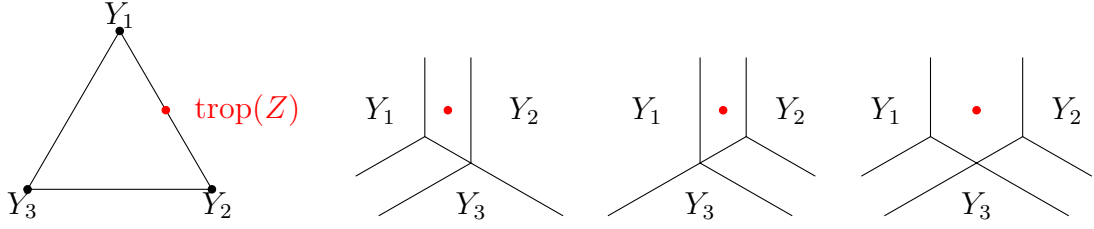

\subsubsection*{The stacks $\Bar{\mathfrak{N}^m_{\SWS}}$ and $\Bar{\mathfrak{N}^m_{\LW}}$.}
Let $S\coloneqq\Spec R\to C$ for some valuation ring $R$ as before and let $\eta$ and $\eta_{0}$ denote the generic and closed points of $S$.
We construct the separated stacks $\Bar{\mathfrak{N}^m_{\SWS}}$ and $\Bar{\mathfrak{N}^m_{\LW}}$ from the stacks $\mathfrak{N}^m_{\SWS}$ and $\mathfrak{N}^m_{\LW}$ by introducing additional equivalences on these stacks. It is sufficient to define these equivalences over the scheme $S$ as we already have a well-defined notion of equivalence up to base change on these stacks.

\medskip
Let $(Z_\eta, \cX_\eta) \in \mathfrak{N}^m_{\SWS}(\eta)$ (or $\mathfrak{N}^m_{\LW}(\eta)$), where all basis directions are invertible at the point $\eta$. Then $(Z_\eta, \cX_\eta)$ will also be an element of  $\Bar{\mathfrak{N}^m_{\SWS}}(\eta)$ (or $\Bar{\mathfrak{N}^m_{\LW}}(\eta)$). We then define to be equivalent in  $\Bar{\mathfrak{N}^m_{\SWS}}$ (or $\Bar{\mathfrak{N}^m_{\LW}}$) all stable (for the respective stability conditions) extensions of this pair. In other words, any two pairs in $\mathfrak{N}^m_{\SWS}(k)$ (or $\mathfrak{N}^m_{\LW}(k)$) which are associated pairs for the same configuration of vertices in $\trop(X_0)$ (or rays in $\trop(X)$) must belong to the same equivalence class in $\Bar{\mathfrak{N}^m_{\SWS}}(k)$ (or $\Bar{\mathfrak{N}^m_{\LW}}(k)$).

\medskip
The only choices here come from the construction of the stack $\mathfrak{X}'$, but all subsequent choices are avoided. This, however, implies that the objects of the stack are equivalence classes of pairs $(Z,\cX)$, where for two representatives $(Z,\cX)$ and $(Z',\cX')$ of a given equivalence class, there may exist no isomorphism of schemes $\cX \cong \cX'$. Defining this more canonical stack therefore comes at the cost of our stack being no longer algebraic. This parallels the canonical choice of underlying stack in Kennedy-Hunt's construction of a logarithmic Quot scheme \cite{K-H}.

\begin{theorem}
The stacks $\Bar{\mathfrak{N}^m_{\SWS}}$ and $\Bar{\mathfrak{N}^m_{\LW}}$ have finite automorphisms and are proper.
\end{theorem}

\begin{proof}
    They have finite automorphisms because the stacks $\mathfrak{N}^m_{\SWS}$ and $\mathfrak{N}^m_{\LW}$ have finite automorphisms by the same argument as Lemma 6.1.10 of \cite{CT}. Passing to the equivalence classes of associated pairs does not add automorphisms, since the equivalences being added are not isomorphisms. The universal closure follows directly from the universal closure of $\mathfrak{N}^m_{\SWS}$ and $\mathfrak{N}^m_{\LW}$ in Proposition \ref{univ closed 2} and Corollary \ref{univ closed 3}, and from Corollary \ref{corollary limits of special elements}. It remains to prove that the stacks are separated.
    
    \medskip
    Let $S \coloneqq \Spec R\to C$, where $R$ is a discrete valuation ring with uniformising parameter $u$. Let $\eta$ denote the generic point of $S$ and $\eta_0$ its closed point. Now, assume that there are two pairs $[(Z,\cX)], [(Z',\cX')]\in \Bar{\mathfrak{N}^m_{\SWS}}(S)$ (or $\Bar{\mathfrak{N}^m_{\LW}}(S)$) such that $[(Z_\eta,\cX_\eta)] \cong [(Z'_\eta,\cX'_\eta)]$. We will show that it must follow that $[(Z_{\eta_0},\cX_{\eta_0})] \cong [(Z'_{\eta_0},\cX'_{\eta_0})]$.
    
    \medskip
    We consider the representatives $(Z,\cX)$ and $(Z',\cX')$ of the above equivalence classes. As in the the proof of Proposition \ref{univ closed 2}, we may assume that $S$ is chosen so that the $i$-th irreducible component of $Z$ is given in terms of its local coordinates $x,y$ and $z$ by
    \begin{equation}
        \{ (c_{i,1}u^{e_{i,1}}, c_{i,2}u^{e_{i,2}}, c_{i,3}u^{e_{i,3}})\},
    \end{equation}
    and the $i$-th irreducible component of $Z'$ is given in terms of its local coordinates $x,y$ and $z$ by
    \begin{equation}
        \{ (d_{i,1}u^{f_{i,1}}, d_{i,2}u^{f_{i,2}}, d_{i,3}u^{f_{i,3}})\}.
    \end{equation}
    Since the equivalences of the stack act trivially on $x,y$ and $z$ and we know that $(Z_\eta,\cX_\eta) \cong (Z'_\eta,\cX'_\eta)$, it must therefore follow that $Z$ and $Z'$ have the same number of irreducible components. Moreover, if these components are labelled in a compatible way, then $c_{i,1} = d_{i,1}$ and $e_{i,1} = f_{i,1}$ for all $i$. This is independent of the choices of representative $Z$ and $Z'$, as the valuation vectors $(e_{i,1},e_{i,2},e_{i,3})$ and $(f_{i,1},f_{i,2},f_{i,3})$ are precisely the data of the tropicalisation of $Z$ and $Z'$ and this is constant across the equivalence classes.

    \medskip
    Now, as before, the fact that $Z$ and $Z'$ must be $S$-flat families in $\cX$ and $\cX'$ respectively means that each $Z_i$ and $Z_i'$ component must satisfy the equations
    \begin{align*}
         x &= c_{i,1}u^{e_{i,1}},\\
        y &= c_{i,2}u^{e_{i,2}},\\
        z &= c_{i,3}u^{e_{i,3}},
    \end{align*}
    also above the closed point. If more than one element of the set $\{ e_{i,1}, e_{i,2}, e_{i,3} \}$ is nonzero, then $Z_i$ and $Z_i'$ must be supported in a component blown up along $\langle x, cu^{e_{i,1}} \rangle$ and $\langle y, c'u^{e_{i,2}} \rangle$ over the closed point $\eta_0$, for some $c,c'\in R^\times$. But this tells us exactly that the nonempty bubbles of the pairs $(Z_{\eta_0},\cX_{\eta_0})$ and $(Z'_{\eta_0},\cX'_{\eta_0})$ must correspond to the same points in the tropicalisation and therefore these two pairs must belong to the same equivalence class.
\end{proof}

\begin{remark}
    As mentioned above, the stacks $\Bar{\mathfrak{N}^m_{\SWS}}$ and $\Bar{\mathfrak{N}^m_{\LW}}$ mirror Kennedy-Hunt's underlying stack construction of a logarithmic Quot scheme. We should note, however, that restricting the construction of \cite{K-H} to the case of Hilbert schemes of points, though similar, would not yield either of the stacks $\Bar{\mathfrak{N}^m_{\SWS}}$ and $\Bar{\mathfrak{N}^m_{\LW}}$. This is because some choices were made in the way we constructed $X[A,B]$, and $\mathfrak{X}'$ does not contain every possible expansion of $X$. Kennedy-Hunt's construction precludes any such choices. Let us denote for now by $\mathfrak{Y}\to C$ the underlying stack of the restriction of the logarithmic Quot scheme construction to the case of Hilbert schemes of points. For any algebraically closed field $k$, an object of $\mathfrak{Y}(k)$ is an equivalence class. While there is a bijection between the sets of equivalence classes $|\mathfrak{Y}(k)|$ and $|\Bar{\mathfrak{N}^m_{\SWS}}(k)|$ or $|\Bar{\mathfrak{N}^m_{\LW}}(k)|$ of the respective stacks, some representatives of the equivalence class defining an object of $\mathfrak{Y}(k)$ do not exist in the equivalence class defining the corresponding object in $\Bar{\mathfrak{N}^m_{\SWS}}(k)$ or $\Bar{\mathfrak{N}^m_{\LW}}(k)$.
\end{remark}

Finally, we would like to comment on the fact that we do not assert an isomorphism between the stacks $\Bar{\mathfrak{N}^m_{\SWS}}$ and $\Bar{\mathfrak{N}^m_{\LW}}$ here. This is because the results we used to establish equivalence of stacks require the stacks to be algebraic. We do, in fact, expect these stacks to be isomorphic but some further work is needed to show this.

\subsection{Choices of proper algebraic stacks.}\label{choices of stacks}

As mentioned before, the second way we have of making separated stacks is to make a systematic choice among all pairs $(Z,\cX) \in \mathfrak{N}^m_{\SWS}(k)$ (or $\mathfrak{N}^m_{\LW}(k)$) associated to the same configuration of vertices in $\trop(X_0)$. For example, in Figure \ref{assoc pairs}, we would need to add a stability condition which allows exactly one of the pairs represented by the three geometric pictures on the right to be stable. These choices effectively cut out proper substacks of $\mathfrak{N}^m_{\SWS}$ and $\mathfrak{N}^m_{\LW}$.

\subsubsection*{Additional stability condition for uniqueness of limits.}
In order to make these choices of limit representatives, we may declare that certain fibres $\cX\in \mathfrak{X}'$ contain no stable subschemes and we may also decide that certain bubbles in given fibres contain no points of the support. The Maulik-Ranganathan stability condition is close to this description. The starting point for their construction is to look at the image of a given subscheme $Z$ under the tropicalisation map, given by a collection of vertices in the tropicalisation of $X_0$, and then construct an expansion around the bubbles introduced by these new vertices, corresponding to our notion of associated pairs. Implicitly, choices are made at this stage to exclude certain superfluous fibres. We must therefore emulate this with our stability condition. Moreover, in the modified special fibres they do construct, it is necessary to allow certain configurations of points and not others, hence labelling certain bubbles as tube components and introducing the concept of Donaldson-Thomas stability with respect to these labellings. We describe suitable additions to the stability conditions previously discussed and extend Maulik-Ranganathan stability to our situation with the following definitions.

\begin{definition}
    Let $\cX\in\mathfrak{X}'(k)$ and let $Z$ be a length $m$ zero-dimensional subscheme in $\cX$. A \emph{tube labelling} of $\cX$ is a choice of a collection of bubbles in $\cX$ which we call \emph{tube components}. Given such a labelling, the pair $(Z,\cX)$ is \textit{Donaldson-Thomas stable} (abbreviated DT stable) if it satisfies the following condition: a bubble in $\cX$ contains no point of the support of $Z$ if and only if it is a tube component.
\end{definition}

\begin{definition}\label{MR stab}
    Let $\alpha$ denote a collection of equivalence classes of fibres in $\mathfrak{X}'(k)$ and let $\beta$ denote a collection of tube labellings on the equivalence classes of fibres remaining after the fibres $\alpha$ have been removed. Let $\cX\in\mathfrak{X}'(k)$. We say that the pair $(Z,\cX)$, where $Z$ is a length $m$ zero-dimensional subscheme in $\cX$, is $(\alpha,\beta)$-\textit{stable} if $\cX$ is not one of the fibres in $\alpha$ and the pair $(Z,\cX)$ is DT stable.
\end{definition}

We will only want to consider $(\alpha,\beta)$-stability as a restriction of SWS or LW stability. We denote by $\mathfrak{N}^m_{\SWS,(\alpha,\beta)}$ and $\mathfrak{N}^m_{\LW,(\alpha,\beta)}$ the stacks $\mathfrak{N}^m_{\SWS}$ and $\mathfrak{N}^m_{\LW}$ restricted to their $(\alpha,\beta)$-stable loci. As we will see with the next two results, an appropriate choice of $(\alpha,\beta)$-stability condition will give us unique limits and allow us to build proper stacks. First, let us define some useful terminology.

\begin{definition}
    We say that a pair $(\alpha,\beta)$ is an \textit{almost proper SWS stability condition} if it is chosen so that for each configuration of points in $\trop(X_0)$ there exists a unique associated pair $(Z,\cX)\in \mathfrak{N}^m_{\SWS,(\alpha,\beta)}(k)$, where $\cX$ is pulled back from a modified special fibre. The corresponding definition of \textit{almost proper LW stability condition} can be made in a similar way.
\end{definition}

\begin{proposition}\label{generic limit}
    Let $S\coloneqq \Spec R$ for some discrete valuation ring $R$ and let $\xi \colon S\to C[A,B]$ be a morphism that maps the generic point $\eta$ of $S$ to a vector in $C[A,B]$ whose entries are all invertible, for some unbroken $(A,B,n)$.
    Now, assume that $(\alpha,\beta)$ is an almost proper SWS or LW stability condition and let $\cX\coloneqq \xi^*X[A,B]$. If $(Z_\eta\cX_\eta)\in \mathfrak{N}^m_{\SWS,(\alpha,\beta)}(\eta)$ (or $\mathfrak{N}^m_{\LW,(\alpha,\beta)}(\eta)$), then $(Z_\eta,\cX_\eta)$ has a unique extension in $\mathfrak{N}^m_{\SWS,(\alpha,\beta)}$ (or $\mathfrak{N}^m_{\LW,(\alpha,\beta)}$).
\end{proposition}

\begin{proof}
    This follows directly from the proof of Proposition \ref{univ closed 2}. Indeed, given such a pair $(Z_\eta,\cX_\eta)$, we recall that, up to some finite base change, each irreducible component of $Z_\eta$ can be written locally in terms of its $x,y$ and $z$ coordinates as
    \begin{equation}
        \{ (c_{i,1}u^{e_{i,1}}, c_{i,2}u^{e_{i,2}}, c_{i,3}u^{e_{i,3}})\}, \nonumber
    \end{equation}
    for some $e_{i,j}\in \ZZ$ and $c_{i,j}\in R^\times$. Moreover, we showed that if more than one element of the set $\{ e_{i,1},e_{i,2},e_{i,3} \}$ is nonzero, the closure of $Z_\eta$ in any stable extension of $(Z_\eta,\cX_\eta)$ must be supported in a component blown up along the vanishing of both sides of the relevant equations \eqref{subscheme equations}. Now, we note that each of these triples $(e_{i,1},e_{i,2},e_{i,3})$ corresponds to a point in $\trop(X_0)$. By the proof of Proposition \ref{univ closed 2}, any associated pair for $\trop(Z_\eta)$ gives an extension of $(Z_\eta,\cX_\eta)$. Since $(\alpha,\beta)$ is an almost proper SWS stability condition, we know that there exists a unique $(\alpha,\beta)$-stable associated pair providing such an extension.
\end{proof}

We have therefore proven that, if $(\alpha,\beta)$ is an almost proper SWS (or LW) stability condition, families whose generic fibre is pulled back from a fibre in $X[A,B]$ over a point $(t_1,\dots,t_{n+1})$ whose entries are all nonzero have a unique $(\alpha,\beta)$-stable limit.

\subsubsection*{Compatibility condition for properness.}
In order to ensure that the stacks $\mathfrak{N}^m_{\SWS,(\alpha,\beta)}$ and $\mathfrak{N}^m_{\LW,(\alpha,\beta)}$ are proper, we need one more compatibility condition to hold. The following example illustrates how the stacks may fail to be proper if we do not add this extra condition.

\medskip
\emph{Example of incompatible limits.} Let $(\alpha,\beta)$ be an almost proper LW stability condition. Let the equivalence class in $\mathfrak{N}^m_{\LW,(\alpha,\beta)}$ associated to $\trop(X_0)$ with one vertex added to the $Y_1\cap Y_2$ edge be given by expanding one $\Delta_1$-component of pure type. Now, we decide that the equivalence class associated to $\trop(X_0)$ with two vertices added to the $Y_1\cap Y_2$ edge will be given by expanding two $\Delta_2$-components of pure type. These determine the unique associated pairs for these tropical configurations. See Figure \ref{incompatible limit}. Now let us take a length 2 zero-dimensional subscheme in the first fibre whose support consists of a point in $\Delta_1^\circ \cap \pi^*(Y_1\cap Y_2)^\circ$ and a point in $Y_2^\circ$. This is LW and $(\alpha,\beta)$-stable for our choice of $(\alpha,\beta)$. We may now study the limit of this subscheme as the point in $Y_2^\circ$ tends towards $\pi^*(Y_1\cap Y_2)$, i.e.\ as its $x$ coordinate tends to zero. If the stack $\mathfrak{N}^m_{\LW,(\alpha,\beta)}$ were proper it would contain this limit, but it does not. Indeed, our second fibre has the same associated tropical configuration as this limit and by the almost proper condition it must be the unique fibre in $\mathfrak{N}^m_{\LW,(\alpha,\beta)}$ up to equivalence which does, but it is clearly not a limit for this pair. This breaks universal closure.

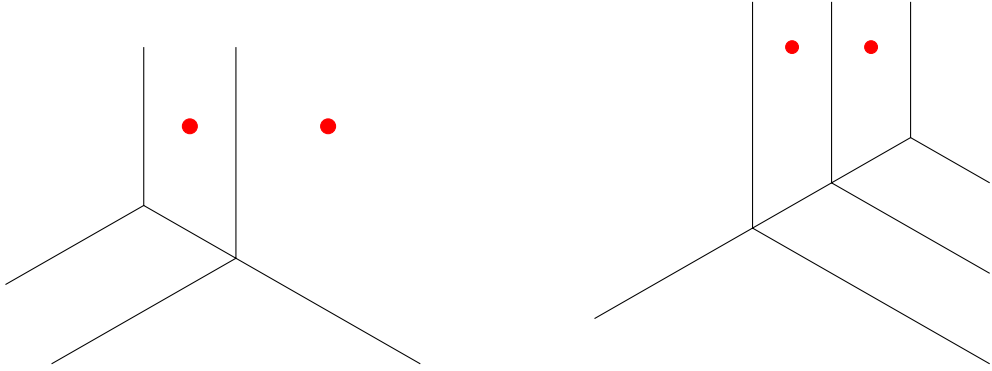
\begin{figure} 
    \begin{center}   
    \begin{tikzpicture}[scale=1.4]
        \draw   (-0.866,0.5) -- (0,0)
        (0,0) -- (0,2)       
        (-0.866,0.5) -- (-0.866,2)
        (-1.732, -1) -- (0,0)
        (-2.165,-0.249) -- (-0.866,0.5)
        
        (1.732, -1) -- (0,0);
        \filldraw[red] (-0.433,1.25) circle (2pt) ;
        \filldraw[red] (0.866,1.25) circle (2pt) ;
    \end{tikzpicture}
    \hspace{2cm}
    \begin{tikzpicture}[scale=1.2]
\draw   

        (0,-1) -- (0.866,-0.5)
        (0,-1) -- (0,1.5)
        
        (0.866,-0.5) -- (1.732,0)
        
        (1.732,0) -- (1.732,1.5)
        
        (2.599, -1.5) -- (0.866,-0.5)
        (0.866,-0.5) -- (0.866, 1.5)
        (2.599, -2.5) -- (0,-1)
        (-1.732,-2) -- (0,-1)
        (2.599, -0.5) -- (1.732,0);
        \filldraw[red] (0.433,1) circle (2pt) ;
        \filldraw[red] (1.299,1) circle (2pt) ;
\end{tikzpicture}
    
    \end{center}
    \caption{The picture on the right is not a limit of the one on the left.}
    \label{incompatible limit}
\end{figure}

\medskip
We recall the following definitions from \cite{CT}.

\begin{definition}\label{tropical equivalence}
    Let S be a $C$-scheme and $\cX\in \mathfrak{X}(S)$. The equivalence class $[\trop(\cX)]$ is defined as follows. For any $C$-scheme $S'$ and any $\cX'\in \mathfrak{X}(S')$, we say $\trop(\cX')$ belongs to the equivalence class $[\trop(\cX)]$ if and only if $\cX$ and $\cX'$ are equivalent in $\mathfrak{X}$.
\end{definition}

By convention, if $\trop(\cX)$ is not well-defined (for example, if we have only specified a family in $X[A,B]$ from which it is pulled back but not a morphism $S\to C[A,B]$), we may take $\trop(\cX)$ to be any element of the equivalence class $[\trop(\cX)]$.

\medskip
For an unbroken pair $(A,B,n)$ and a set $I\subseteq [n+1]$, let $X[A,B]_I$ denote the restriction of $X[A,B]$ to the locus where $t_i=0$ for all $i\in I$. 

\begin{definition}
    Let $(Z_\eta,\cX_\eta)\in \mathfrak{N}^m_{\SWS,(\alpha,\beta)}(\eta)$ (or $\mathfrak{N}^m_{\LW,(\alpha,\beta)}(\eta)$) be any pair over the generic point of some $S\coloneqq \Spec R$, for some discrete valuation ring $R$ as before. Moreover, let $\cX_\eta$ be a restriction of $\cX\coloneqq \xi^* X[A,B]_I$ for some nonempty set $I$, i.e.\ $\cX_\eta$ is pulled back from some modified special fibre.  If, for any associated pair $(Z'_{\eta'_0}, \cX'_{\eta'_0})$ of $\trop(Z_\eta)$ in $\mathfrak{N}^m_{\SWS}$ (or $\mathfrak{N}^m_{\LW}$), the tropicalisation $\trop(\cX'_{\eta'_0})$ is a subdivision of a representative of the equivalence class $[\trop(\cX_\eta)]$ (as in Definition \ref{tropical equivalence}), then we say that $\mathfrak{N}^m_{\SWS,(\alpha,\beta)}$ (or $\mathfrak{N}^m_{\LW,(\alpha,\beta)}$) is \emph{tropically compatible}.
\end{definition}

Now, we may define the appropriate compatibility condition for $(\alpha,\beta)$-stability.

\begin{definition}\label{proper SWS stab}
    Let $(\alpha,\beta)$ define an almost proper SWS stability condition. If the corresponding stack $\mathfrak{N}^m_{\SWS,(\alpha,\beta)}$ is tropically compatible, we say that $(\alpha,\beta)$ defines a \textit{proper SWS stability condition} (abbreviated PSWS stability). A similar definition can be made for LW stability.
    If $(\alpha,\beta)$ is a proper LW stability condition and the pair $(Z,\cX)$ is LW and $(\alpha,\beta)$-stable, then we say it is \textit{Maulik-Ranganathan stable} (abbreviated MR stable) for the given choice $(\alpha,\beta)$.
\end{definition}

\begin{proposition}\label{proper stab proposition}
    Let $(\alpha,\beta)$ be some choice of stability condition. The corresponding stack $\mathfrak{N}^m_{\SWS,(\alpha,\beta)}$ is proper if and only if $(\alpha,\beta)$ is a PSWS stability condition.
\end{proposition}

\begin{proof}
This follows from Proposition 6.1.8 of \cite{CT}, since we saw in the proof of that proposition that there are two conditions needed to show the moduli stack is proper:
\begin{itemize}
    \item existence and uniqueness of the associated pair, which is given two us by the almost proper SWS stability condition $(\alpha,\beta)$, and
    \item the tropically compatible property, which is given to us by the fact that $(\alpha,\beta)$ is moreover a proper SWS stability condition.
\end{itemize}
\end{proof}

\begin{corollary}\label{proper stab corollary}
The stack $\mathfrak{N}^m_{\LW,(\alpha,\beta)}$ is proper if and only if $(\alpha,\beta)$ is a proper LW stability condition.    
\end{corollary}

\begin{proof}
    This follows immediately from the proof of Proposition \ref{proper stab proposition}.
\end{proof}

\emph{Notation.} In the following we shall denote by $\mathfrak{N}^m_{\PSWS,(\alpha,\beta)}$ a stack $\mathfrak{N}^m_{\SWS,(\alpha,\beta)}$ where $(\alpha,\beta)$ is a proper SWS stability condition, and by $\mathfrak{N}^m_{\MR,(\alpha,\beta)}$ a stack $\mathfrak{N}^m_{\LW,(\alpha,\beta)}$ where $(\alpha,\beta)$ is a proper LW stability condition.

\medskip
The stacks $\mathfrak{M}^m_{\SWS}$ and $\mathfrak{M}^m_{\LW}$ from \cite{CT} are examples of stacks $\mathfrak{N}^m_{\PSWS,(\alpha,\beta)}$ and $\mathfrak{N}^m_{\MR,(\alpha,\beta)}$. Indeed, the choice $(\alpha,\beta)$ which defines these stacks clearly satisfies the conditions of a proper SWS (or LW) stability condition.

\begin{corollary}\label{corollary limits of special elements}
    Let $S$ and $\eta$ as before and let $(Z_\eta,\cX_\eta)$ be an element of $\mathfrak{N}^m_{\SWS}(\eta)$, $\mathfrak{N}^m_{\LW}(\eta)$, $\Bar{\mathfrak{N}^m_{\SWS}}(\eta)$ or $\Bar{\mathfrak{N}^m_{\LW}}(\eta)$ such that $\cX_\eta$ is a modified special fibre. Then $(Z_\eta,\cX_\eta)$ has a stable extension with respect to the relevant stability condition in the relevant stack.
\end{corollary}

\begin{proof}

    In the non-separated stacks $\mathfrak{N}^m_{\SWS}$ and $\mathfrak{N}^m_{\LW}$, we allow for all limits. In particular, each stack will contain all the associated pairs for $\trop(Z_\eta)$ which are compatible with $\cX_\eta$ and stable for their respective stability conditions.

    \medskip In the case of $\Bar{\mathfrak{N}^m_{\SWS}}(\eta)$ and $\Bar{\mathfrak{N}^m_{\LW}}(\eta)$, we do not make any choices, so for any such $(Z_\eta,\cX_\eta)$ all possible associated pairs of $\trop(Z_\eta)$ form an equivalence class in $\Bar{\mathfrak{N}^m_{\SWS}}(\eta)$ or $\Bar{\mathfrak{N}^m_{\LW}}(\eta)$. In particular, there will be a representative of this equivalence class which is compatible with $\cX_\eta$ in the sense of Definition \ref{proper SWS stab}.
\end{proof}

\begin{theorem}\label{theorem choice proper DM}
    The stacks $\mathfrak{N}^m_{\PSWS,(\alpha,\beta)}$ and $\mathfrak{N}^m_{\MR,(\alpha,\beta)}$ are Deligne-Mumford and proper over $C$.
\end{theorem}

\begin{proof}
    Properness follows from Proposition \ref{proper stab proposition} and Corollary \ref{proper stab corollary}. They are Deligne-Mumford because SWS and LW stability ensure finite automorphisms.
\end{proof}

\subsubsection*{Remark
about compatibility.} 
Here, we point out a subtlety about the compatibility conditions discussed above. Let $\cX_1$ and $\cX_2$ be two fibres of $\mathfrak{X}'$ over a closed point. Assume that $\cX_1$ is pulled back from some $X[A,B]_{J_1}$ and $\cX_2$ is pulled back from some $X[A',B']_{J_2}$, where $|J_1|\leq |J_2|$. We denote by $W_1\subseteq \cX_1$ and $W_2\subseteq \cX_2$ the union of all irreducible components in each fibre which are not tubes (this includes the $Y_i$ components). We say that $(\alpha,\beta)$ defines a \textit{strong compatibility condition} if the following holds. If $\pi_*(W_2)\subset\Bar{\pi_*(W_1)}\subset X_0$, then there exists a fibre $\cX_3$ equivalent to $\cX_2$ in $\mathfrak{X}'$ such that $\cX_3$ is pulled back from $X[A,B]_{J_3}$ for some $J_3$ with $J_1\subseteq J_3$. This condition certainly implies tropical compatibility, and naively may seem like the right restriction to impose, but it is actually stronger than necessary. Indeed, the tropically compatible condition does not necessarily imply strong compatibility. As shown in Remark \ref{don't need strong compatibility}, strong compatibility is not necessary to construct a proper stack, as in certain cases, subschemes cannot tend towards a certain limit.

\begin{remark}\label{don't need strong compatibility}
    Let $(Z_\eta,\cX_\eta)$ be a pair consisting of a length $m$ zero-dimensional subscheme $Z_\eta$ in a modified special fibre $\cX_\eta$. Assume that the pair is either SWS or LW stable. Let $P$ be a point of the support which lies in some irreducible component $W\subset \cX_\eta$ and let $V$ be some other irreducible component in $\cX_\eta$ which intersects $W$ nontrivially. If there exists no representative in the equivalence class of $(Z_\eta,\cX_\eta)$ such that there is a smoothing from the interior of $W$ in $\cX_\eta$ to the interior of some $\Delta$-components expanded out in the $W\cap V$ locus, i.e.\ such that these $\Delta$-components are equal to $W$ in $\cX_\eta$, then there exists no flat family of length $m$ zero-dimensional subschemes such that $(Z_\eta,\cX_\eta)$ is the generic fibre and $P$ tends towards $W\cap V$ over the closed point. This can be  seen by studying the equations of the blow-ups.
    
    For example, let $\cX_\eta$ be a fibre with one expanded $\Delta_1$-component of pure type and no other expanded components. Let $Z_\eta$ be a length 2 zero-dimensional subscheme with one point of its support, $P_0$, lying in the interior of this $\Delta_1$-component in the $\pi^*(Y_1\cap Y_2)^\circ$ locus and the other point of the support, $P_1$, lying in $Y_2^\circ$, as in the picture on the left of Figure \ref{incompatible limit}. In this fibre, there cannot be a $\Delta_2$-component which is equal to $Y_2$ and thus there can be no smoothing from the interior of $Y_2$ to an expanded component in the $Y_2\cap Y_3$ locus. But there also cannot be any flat family such that $P_1$ tends towards $z=0$, as any equation for $P_1$ fixes $z$ and does not set it to be proportional to any variable which may tend towards zero (like a $y_1^{(k)}/y_0^{(k)}$ or $t_i$ variable).
\end{remark}

In Section \ref{smoothings section}, we give a tropical criterion to understand when such smoothings from one component to another exist.

\subsection{Relating the proper stacks}\label{2nd isomorphism of stacks}

We may now show how the stacks $\mathfrak{N}^m_{\PSWS,(\alpha,\beta)}$ and $\mathfrak{N}^m_{\MR,(\alpha,\beta)}$ relate to each other.

\begin{theorem}
    Let $\mathfrak{N}^m_{\PSWS,(\alpha,\beta)}$ and $\mathfrak{N}^m_{\MR,(\alpha,\beta)}$ be two choices given by the same $(\alpha,\beta)$. This means that we have chosen an MR stability condition $(\alpha,\beta)$ which also defines a PSWS stability condition (by Proposition \ref{SWS stab 2} this is not the case in general). There then exists an isomorphism of stacks
    \[
    \mathfrak{N}^m_{\PSWS,(\alpha,\beta)}\cong \mathfrak{N}^m_{\MR,(\alpha,\beta)}.
    \]
\end{theorem}

\begin{proof}
    Similarly to the arguments of Section 6.2 of \cite{CT}, it is clear that on the scheme level we have an inclusion
    \[
    H^{m,[A,B]}_{\SWS,(\alpha,\beta)} \longhookrightarrow H^{m,[A,B]}_{\LW,(\alpha,\beta)},
    \]
    where $H^{m,[A,B]}_{\SWS,(\alpha,\beta)}$ and $H^{m,[A,B]}_{\LW,(\alpha,\beta)}$ are the restrictions of $H^m_{[A,B],\SWS}$ and $H^m_{[A,B],\LW}$ to their respective $(\alpha,\beta)$-stable loci. As before, this gives rise to a morphism of stacks
    \[
    f\colon \mathfrak{N}^m_{\PSWS,(\alpha,\beta)}\longrightarrow \mathfrak{N}^m_{\MR,(\alpha,\beta)}.
    \]
    By our assumptions and Theorem \ref{stability theorem}, for every MR stable pair $(Z,\cX)\in \mathfrak{N}^m_{\MR,(\alpha,\beta)}(k)$, where $k$ is an algebraically closed field as before, there exists a representative of the equivalence class of $(Z,\cX)$ which is SWS stable. We therefore have an induced bijection between the sets of equivalence classes of each stack
    \[
    |f| \colon |\mathfrak{N}^m_{\PSWS,(\alpha,\beta)}(k)| \longrightarrow |\mathfrak{N}^m_{\MR,(\alpha,\beta)}(k)|.
    \]
    By Theorem \ref{theorem choice proper DM}, $\mathfrak{N}^m_{\PSWS,(\alpha,\beta)}$ is separated and Deligne-Mumford and therefore has finite inertia. The morphism $f$ induces a bijective homomorphism of stabilisers by the same argument as in the proof of Theorem 6.2.4 of \cite{CT}. By Lemma 6 of \cite{AK} (Lemma 6.2.3  of \cite{CT}), this implies that $f$ is representable. Therefore $f$ is an isomorphism of stacks by Lemma 6.2.1 of \cite{CT}.
\end{proof}

As shown in \cite{MR}, all choices of proper Deligne-Mumford stacks obtained for different $(\alpha,\beta)$, are birational to each other.

\subsection{Further comments on smoothings}\label{smoothings section}

Given a pair $(Z,\cX)\in \mathfrak{N}^m_{\SWS}(k)$ (or $ \mathfrak{N}^m_{\LW}(k)$), where $\cX$ is pulled back from a modified special fibre, we would like to understand what deformations of $(Z,\cX)$ exist in $\mathfrak{N}^m_{\SWS}$ (or $ \mathfrak{N}^m_{\LW}$). In particular we would like to study what smoothings exist from irreducible components of $\cX$ to irreducible components of a second fibre $\cX'$. This will allow us to understand explicitly what limits need to be included in the stacks we construct. We will say there is a \emph{smoothing} from an irreducible component $W\subset \cX$ to an irreducible component $W'\subset \cX'$ if there exist a discrete valuation ring $R$ and a smooth variety $\cY$ over $S\coloneqq\Spec R$ such that $W$ is equal to the restriction of $\cY$ to the generic point of $S$ and the restriction of $\cY$ to the closed point of $S$ is a union of irreducible components, where $W'$ is one of these components. The question we seek to answer in this subsection is the following.

\medskip
\emph{Given a polyhedral subdivision of $\trop(X_0)$, what smoothings exist from the components corresponding to the vertices in $\trop(X_0)$}?

\medskip
The question of smoothings takes place on the level of representatives of equivalence classes in $\mathfrak{X}'$ as it depends on $\Delta_i$-multiplicity (see Definition \ref{delta multiplicity}). In order to interpret this information tropically, we enhance $\trop(X_0)$ with the notion of $\Delta$-multiplicity.

\begin{definition}
    Let $\cX\in\mathfrak{X}'(k)$ be a modified special fibre and let $\trop(\cX)$ be the corresponding polyhedral subdivision of $\trop(X_0)$. For each edge of $\Delta$-type 1 or 2 in $\trop(\cX)$ we may assign to it a positive integer, called its \emph{tropical $\Delta$-multiplicity}.
\end{definition}

\begin{remark}
    The tropical $\Delta$-multiplicity of an edge in $\trop(\cX)$ is the $\Delta_i$-multiplicity common to all geometric components corresponding to the vertices connected to this edge. For convenience, here we assign the multiplicity to the edge and not the vertices, as the edge contains the information of the $\Delta$-type (see Definition \ref{delta multiplicity}).
\end{remark}

Given a fibre $\cX$ as above and any bubble $B_0$ in $\cX$, we will show in Proposition \ref{smoothings proposition} what smoothings exist from $B_0$ to bubbles $B_i$ in blow-ups of $\cX$ (possibly after base change). Before formulating this proposition, we describe the operations on $\trop(\cX)$ which can result in smoothings.

\subsubsection*{Allowing for base changes.} As we are working with stacks, we always allow for base changes. Tropically, this means we may adjust the height of the subdivided triangle $\trop(\cX)$ within the cone of the 1-parameter family. We abuse notation slightly and denote this object after base change by $\trop(\cX)$ again. Effectively, this re-scales and allows us to add more integral points to the triangle. Geometrically, this corresponds to the phenomenon that we can consider any scheme $X[A,B]$ to be a sublocus of a larger scheme $X[A',B']$, where the embedding is given by \eqref{AB embed}. This operation of placing a family within a larger family is exactly what allows us to make all appropriate modifications to $X_0$ in the expanded degeneration setup. We may describe blow-ups of $\cX$ by adding edges to the re-scaled triangle which create vertices by intersecting at these new integral points.

\subsubsection*{Assigning tropical $\Delta$-multiplicity.} The next step is to assign a tropical $\Delta$-multiplicity to each edge of $\trop(\cX)$. By definition, this must be greater or equal to 1. It does not affect the $\Delta$-type of the vertices attached to this edge. Let $e$ be an edge of $\Delta$-type $1$ and multiplicity $k$ attached to the vertex $(a,b,0)$ of the triangle (where the coordinates are given in the 3-dimensional cone). Let $b'$ be the smallest integer such that $b'>b$ and such that $(a',b',0)$ has an edge attached for some $a'$. Let $b''$ be the largest integer such that $b>b''$ and such that $(a'',b'',0)$ has an edge attached for some $a''$. We remark here that the $\Delta$-multiplicity $k$ of $e$ does not need to equal the number of integral points up until the next edge, i.e.\ $k$ is independent of the values $b'-b$ or $b-b''$. A similar statement can be made for $e$ of $\Delta$-type $2$.

Note that whereas on the level of schemes, $\Delta$-multiplicity makes an important difference, on the level of stacks varying the $\Delta$-multiplicity preserves the equivalence class of the fibre $\cX$. Even though here we work with stacks, it is still important to consider the underlying schemes in order to understand what smoothings exist.

\subsubsection*{Sliding the edges.} Where edges have tropical $\Delta$-multiplicity greater than 1, we can slide out as many copies of this edge as the number of the multiplicity, as long as the new configuration is compatible with the unbroken condition. These edges cannot cross another parallel edge. Keeping the notation as above, let $e'$ denote the edge attached to $(a',b',0)$ and let $e''$ be the edge attached to $(a'',b'',0)$. These may be of type 1 or 2. Sliding edges from $e$ can be described formally as defining a new unbroken configuration where there are $r$ edges between $e'$ and $e''$, including $e$, and the tropical $\Delta$-multiplicities of these $r$ edges should sum to $k$. The unbroken condition imposes some restrictions on how the edges can slide. The situation is similar if we start with an edge $e$ of type 2, or we may simultaneously slide edges of type 1 and 2 which are attached to the same point of $Y_1\cap Y_2$.

Wherever one of these new edges which slides out from $e$ intersects another edge, it creates a new vertex, which corresponds to a new bubble on the geometric side.

\subsubsection*{Criterion for smoothability.} Let $\cX$ and $B_0$ as above. Let $v_0$ be the vertex in $\trop(\cX)$ corresponding to $B_0$ and let $\{e_j\}$ be the edges of $\Delta$-type 1 and 2 going through this vertex. There are at most two and each one has a given multiplicity. Let $\{v_i\}_{i\geq 1}$ be the collection of all vertices resulting from intersections which arise when the copies of the edges $e_i$ are allowed to slide to an unbroken configuration in the way described above. Moreover, we include in the collection $\{v_i\}_{i\geq 1}$ only vertices which share an edge with $v_0$. Finally, let $\{B_i\}_{i\geq 1}$ be the geometric bubbles corresponding to these vertices.

\begin{proposition}\label{smoothings proposition}
    There is a smoothing from $B_0$ to each of the $B_i$ and there exist no other smoothings from $B_0$. Therefore, a subcheme $Z$ supported in $B_0$ is either fixed or tends towards the interior of a component $B_i$.
\end{proposition}

\begin{proof}
    In étale local coordinates, the bubble $B_0$ is given by the equations \eqref{BU1 eqns} describing the family $X[A,B]$, as well as some additional equations which fix certain variables in order to determine $B_0$ within the family. We consider the nonzero equations defining $B_0$. Each of these nonzero equations necessarily contains coordinates corresponding to some $\PP^1$ associated to a $\Delta$-component. The variables in these equations may vary as long as both sides of the equation vary together so that the equalities are preserved. Note that if the base codimension is greater or equal to one, then varying the values on both sides of the equation preserves the equivalence class of $\cX$ in $\mathfrak{X}'$. The equivalence class changes only when both sides go to zero, in which case the new equivalence class this gives is in the closure of the first one. A bubble $B'$ has a smoothing from $B_0$ if and only if its equations can be obtained from those of $B_0$ by letting both sides of some nonzero equations of $B_0$ tend to zero. Therefore, all equations which define $B_0$ must hold for $B'$, where some nonzero equations of $B_0$ may go to zero for $B'$. In particular, $B'$ cannot be determined by a nonzero equation which does not hold for $B_0$. Equivalently, smoothings do not allow us to enter $\Delta$-components in which we were not before.

    This is exactly what the set of bubbles $\{B_i\}_{i\geq 1}$ defined above gives us. They correspond to vertices $v_i$ created by sliding out and intersecting copies of the edges of $\Delta$-types 1 and 2 passing through $v_0$. Since each vertex $v_i, i\geq 1$, must share an edge with $v_0$, the set of $\Delta_j^{(k)}$-components which are equal to the corresponding bubble $B_i$ must be a subset of the set of $\Delta_j^{(k)}$-components which are equal to $B_0$. It follows that no nonzero equation can hold for a bubble in $\{B_i\}_{i\geq 1}$ which does not hold for $B_0$. Any vertex which is not created in this way or does not share an edge with $v_0$ will either correspond to a bubble equal to a $\Delta_j^{(k)}$-component which is not equal to $B_0$ and will therefore be defined by some nonzero equation which does not hold for $B_0$; or will correspond to a component which is not equal to any $\Delta_j^{(k)}$, which means it must be a $Y_i$ component, and therefore one of the equations $x=0$, $y=0$ or $z=0$ which held for $B_0$ does not hold for it.
\end{proof}

\emph{Examples of smoothings.} We can consider a triangle with added edges and vertices defining some $\trop(\cX)$ like that given on the left of Figure \ref{tropical smoothings figure}. Here, we have chosen some vertex which we label $v_0$ and we label the edges going through this vertex by $e_1$ and $e_2$. If, for example, the edge $e_1$ has tropical $\Delta$-multiplicity greater than 1, then we may slide out an edge $e'_1$ as shown in blue on the right hand side of Figure \ref{tropical smoothings figure}. The tropical $\Delta$-multiplicities of $e_1$ and $e'_1$ in the second picture should add up to the tropical $\Delta$-multiplicity of $e_1$ in the first picture. The new vertices created by intersecting the edge $e'_1$ with the existing edges of $\trop(\cX)$ are drawn in blue. We can then see that there exists a smoothing from $v_0$ to $v_1$, as they share an edge, but no smoothings from $v_0$ to $u$ or $w$.

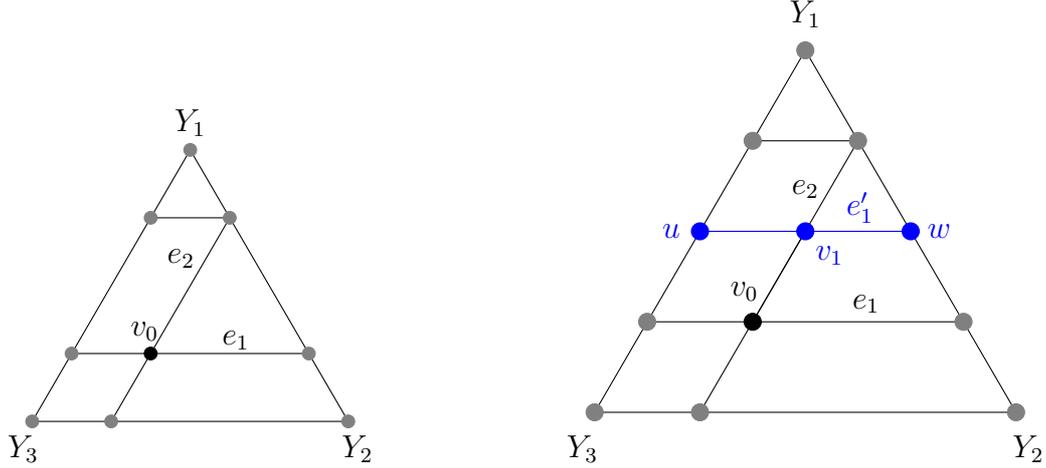
\begin{figure} 
    \begin{center}   
    \begin{tikzpicture}[scale=1.2]
        \draw   

        (-1.732, -1) -- (0,2)       
        (-1.732, -1) -- (1.732, -1)
        
        (1.732, -1) -- (0,2)
        
        ;
        \draw[black] (-0.433,1.25) -- (0.433,1.25);
        \draw[black] (-0.866,-1) -- (0.433,1.25);
        \draw[black] (-1.3,-0.25) -- (1.3,-0.25);
        
        \draw (-0.5,0) node[anchor=center, color = black]{$v_0$};
        \draw (0.5,-0.1) node[anchor=center, color = black]{$e_1$};
        \draw (-0.1,0.8) node[anchor=center, color = black]{$e_2$};
        \filldraw[gray] (-1.3,-0.25) circle (2pt) ;
        \filldraw[gray] (1.3,-0.25) circle (2pt) ;
        \filldraw[gray] (-0.866,-1) circle (2pt) ;
        \filldraw[gray] (0,2) circle (2pt) ;
        \filldraw[gray] (-1.732, -1) circle (2pt) ;
        \filldraw[gray] (1.732, -1) circle (2pt) ;
        \filldraw[gray] (-0.433,1.25) circle (2pt) ;
        \filldraw[gray] (0.433,1.25) circle (2pt) ;
        \filldraw[black] (-0.433,-0.25) circle (2pt) ;
        \draw (0,2.3) node[anchor=center]{$Y_1$};
        \draw (-1.832, -1.3) node[anchor=center]{$Y_3$};
        \draw (1.832, -1.3) node[anchor=center]{$Y_2$};
        
    \end{tikzpicture}
    \hspace{2cm}
    \begin{tikzpicture}[scale=1.6]
\draw   

        (-1.732, -1) -- (0,2)       
        (-1.732, -1) -- (1.732, -1)
        
        (1.732, -1) -- (0,2)
        
        ;
        \draw[black] (-0.433,1.25) -- (0.433,1.25);
        \draw[black] (-0.866,-1) -- (0.433,1.25);
        \draw[black] (-1.3,-0.25) -- (1.3,-0.25);
        \draw[blue] (-0.866,0.5) -- (0.866,0.5);
        \draw (-0.433,-0.25) -- (0,0.5);
        
        \draw (-0.5,0) node[anchor=center, color = black]{$v_0$};
        \draw (0.5,-0.1) node[anchor=center, color = black]{$e_1$};
        \draw (0,0.85) node[anchor=center, color = black]{$e_2$};
        \draw (0.45,0.7) node[anchor=center, color = blue]{$e'_1$};
        \draw (0.2,0.3) node[anchor=center, color = blue]{$v_1$};
        \draw (1.1,0.5) node[anchor=center, color = blue]{$w$};
        \draw (-1.1,0.5) node[anchor=center, color = blue]{$u$};
        \filldraw[gray] (-1.3,-0.25) circle (2pt) ;
        \filldraw[gray] (1.3,-0.25) circle (2pt) ;
        \filldraw[gray] (-0.866,-1) circle (2pt) ;
        \filldraw[gray] (0,2) circle (2pt) ;
        \filldraw[gray] (-1.732, -1) circle (2pt) ;
        \filldraw[gray] (1.732, -1) circle (2pt) ;
        \filldraw[gray] (-0.433,1.25) circle (2pt) ;
        \filldraw[gray] (0.433,1.25) circle (2pt) ;
        \filldraw[black] (-0.433,-0.25) circle (2pt) ;
        \filldraw[blue] (-0.866,0.5) circle (2pt) ;
        \filldraw[blue] (0.866,0.5) circle (2pt) ;
        \filldraw[blue] (0,0.5) circle (2pt) ;
        \draw (0,2.3) node[anchor=center]{$Y_1$};
        \draw (-1.832, -1.3) node[anchor=center]{$Y_3$};
        \draw (1.832, -1.3) node[anchor=center]{$Y_2$};
        
\end{tikzpicture}
    \end{center}
    \caption{Tropical representation of smoothings.}
    \label{tropical smoothings figure}
\end{figure}

\section{Hilbert schemes of points on K3 surfaces and MMP}\label{minimality section}

Up until now, we have only been considering the property that $X\to C$ is semistable. In this section, we will discuss what properties our constructions have if we start with some added assumptions. We will want to show that the proper stacks we construct are semistable themselves, in the following sense.

\begin{definition}\label{semistable dlt stack}
    Let $\mathfrak{Y}$ be a Deligne-Mumford stack which is flat and locally of finite type over $C$. For any $C$-scheme $S$ with a (non-constant) étale morphism $\xi\colon S\to C$ and any $Y\in \Ob(\mathfrak{Y}_C)$ over $C$, we denote the composition map $f\colon \xi^*Y\to C$. We say that $\mathfrak{Y}\to C$ is \emph{semistable} if every such composition map $f$ is semistable.
\end{definition}

Moreover, if $K_X+(X_0)_{\red}$ is semi-ample, i.e.\ $X\to C$ is a \emph{good semistable model}, we would like to show that the proper stacks constructed in previous sections also satisfy this property, which we extend to stacks in the following way.

\begin{definition}\label{good minimal stack}
    Let $\mathfrak{Y}$ be a Deligne-Mumford stack which is flat and locally of finite type over $C$. For any $C$-scheme $S$ with a (non-constant) étale morphism $\xi\colon S\to C$ and any $Y\in \Ob(\mathfrak{Y}_C)$ over $C$, let $W\coloneqq \xi^*Y$ and denote by $f\colon W\to C$ the composition map. We say that $\mathfrak{Y}\to C$ is a \emph{good minimal model} if for every such $W$ and $f$, the divisor $K_{W/C} + f^{-1}(0)$ is semi-ample.
\end{definition}

\medskip
In particular, we will want to consider the case where $X\to C$ is a good type III degeneration of K3 surfaces in the sense of Kulikov \cite{Kulikov}. This means that $\pi\colon X\to C$ is semistable, that its general fibres are smooth K3 surfaces and that the dual complex of its central fibre $X_0$ is a triangulated sphere. Moreover, there exists a relative logarithmic 2-form $\omega_\pi\in H^0(X,\Omega^2_{X/C}(\log X))$ such that $\wedge^n \omega_\pi \in H^0(X,K_{X/C})$ is nowhere vanishing. In this case, we will show that such a 2-form on $X\to C$ induces a nowhere degenerate logarithmic 2-form on the proper stacks we constructed.

\medskip
Most of the results in this section follow easily from the results of Gulbrandsen, Halle, Hulek and Zhang in \cite{GHHZ}. For clarity, however, we give certain of the relevant proofs again here, where the details are slightly different for our situation.

\subsection{Good semistable and dlt minimal models}

\begin{proposition}\label{semistable proposition}
    The stacks $\mathfrak{M}^m_{\SWS}$, $\mathfrak{M}^m_{\LW}$, $\mathfrak{N}^m_{\PSWS,(\alpha,\beta)}$ and $\mathfrak{N}^m_{\MR,(\alpha,\beta)}$ are semistable degenerations over $C$. Moreover, they are normal and $\QQ$-factorial.
\end{proposition}

\begin{proof}
    Semistability follows directly from Lemma 7.3 of \cite{GHHZ}. The stacks are normal as they are semistable. This also applies to any constructions arising from the methods of Maulik and Ranganthan \cite{MR}. Finally, the stacks are $\QQ$-factorial, as all elements of the stacks have finite stabilisers.
\end{proof}

We will now assume that $K_X+(X_0)_{\red}$ is semi-ample and prove the following result.

\begin{theorem}
    The stacks listed in Proposition \ref{semistable proposition} are good minimal models in the sense of Definition \ref{good minimal stack}.
\end{theorem}

\begin{proof}
    Let $S\to C$ be a scheme over $C$ and let $\cX\in \mathfrak{X}'(S)$ (note that $\cX\in \mathfrak{X}(S)$ is a special case), where $\cX \coloneqq \xi^* X[A,B]$ for some étale morphism $\xi \colon S\to C[A,B]$. We denote by $\phi\colon \xi^* X[A,B] \to X[A,B]$ the corresponding strongly cartesian morphism. Now let $P_1,\dots,P_m$ be a collection of points in $\cX$. Since $\phi$ is a base change morphism and $\pi\colon X[A,B] \to X$ is the $G$-equivariant projection, similarly to Lemma 5.12 of \cite{GHHZ}, we have that $K_{\cX} \cong (\pi \circ \phi)^* K_X$ (note that all blow-ups we make in our construction are also small).

    \medskip
    Now, let $P_1,\dots,P_m$ be a finite collection of points in $\cX$. Similarly to Lemma 5.13 of \cite{GHHZ}, we may find a $G$-invariant section of $K_{\cX}^{\otimes r}$, not vanishing at any $P_i$. Indeed, as we have assumed that $X\to C$ is a good semistable minimal model, for each $i$, we know that there exists a section $\Tilde{s_i}$ of $K_{X}^{\otimes r}$ that does not vanish at $(\pi \circ \phi)(P_i)$. It follows by the above that $s_i\coloneqq (\pi \circ \phi)^* \Tilde{s_i}$ is a $G$-invariant section of $K_{\cX}^{\otimes r}$ which does not vanish at $P_i$. Therefore, as in Lemma 5.13 of \cite{GHHZ}, if we take a $k$-linear combination $s = \sum_{i}\lambda_i s_i$, for sufficiently general $\lambda_i\in k$, the section $s$ does not vanish at any $P_i$.

    \medskip
    Now, let $Z = (P_1,\dots,P_m)\in \cX^m$, where $\cX^m$ denotes the product $\cX \times_S \cdots \times_S \cX$. Then, similarly to Lemma 5.14 of \cite{GHHZ}, there exists a $G$- and $S_m$-invariant section $\sigma$ of $(K_{\cX^m})^{\otimes r}$ that does not vanish at $Z$, where $S_m$ denotes the $m$-th symmetric group. This is given by the restriction to the stable locus of the tensor product $\pr_1^*(s) \otimes \cdots \otimes \pr_m^*(s)$ of the pullbacks of $s$ along each projection $\pr_i\colon \cX^m \to \cX$. Clearly, this is $G$- and $S_m$-invariant and does not vanish at $Z$.

    \medskip
    We may now apply a generalised version of Beauville's argument from \cite{Beauville}. Let $W$ denote the inverse image of $X\setminus \Sing(X_0)$ by the composition of morphisms $\pi \circ \phi$. (Note that this is not defined quite in the same way as in \cite{GHHZ} but the overall idea is the same.) We then denote by $W^m$ and $W^{[m]}$ the relative product and relative Hilbert scheme of $W\to S$. Additionally, we denote by $W_*^m$ and $W_*^{[m]}$ the restrictions of the above schemes to the open loci whose points are length $m$ zero-dimensional subschemes which have at most one double point. The complements of $W_*^m$ in $W^m$ and $W_*^{[m]}$ in $W^{[m]}$ are both of codimension 2.

    \medskip
    Now, the Hilbert scheme $W^{[m]}$ is obtained from $W^m$ by blowing up the diagonal in $W^m$ and quotienting by the action of $S_m$. Denote by $b_{\diag}\colon \Tilde{W^m_*} \to W^m_*$ and $q_m\colon \Tilde{W^m_*} \to W^{[m]}_*$ the appropriate restrictions of the blow-up and quotient maps and by $q_G\colon W^{[m]}_* \to W^{[m]}_{*,G}$ the quotient by the group $G$. Let $E$ be the exceptional divisor of $b_{\diag}$. Then, following \cite{Beauville} and \cite{GHHZ}, $E$ is precisely the ramification locus of $q_m$, which gives rise to the isomorphisms
    \begin{align*}
        K_{\Tilde{W^m_*}} &\cong q^*_m(K_{W^{[m]}_*})(E) \\
        K_{\Tilde{W^m_*}} &\cong b^*_{\diag}(K_{W^{m}_*})(E).
    \end{align*}
    Note that if the blow-ups made in the construction of $X[A,B]$ are small, then $q_G$ contracts no divisors and we have again an isomorphism
    \[
    K_{W^{[m]}_*} \cong q_G^*(K_{W^{[m]}_{*,G}}).
    \]
    This, in turn, will yield an isomorphism
    \[
    (q_G \circ q_m)*(K_{W^{[m]}_{*,G}})(E) \cong b^*_{\diag}(K_{W^{m}_*})(E).
    \]
    Finally, let $W^{[m],s}_{*,G}$ be the restriction of $W^{[m]}_{*,G}$ to some stable locus, with respect to any of the stability conditions defined in previous sections. For any point in $W^{[m],s}_{*,G}$, the section $(q_G\circ q_m)_* b_{\diag}^* \sigma$, where $\sigma$ is defined as above gives an everywhere nonvanishing section of $(K_{W^{[m],s}_{*,G}})^{\otimes r}$. But as $W^{[m],s}_{*,G}$ has codimension 2 in $W^{[m],s}_{G}$, the section $\sigma$ extends to a nonvanishing section of $(K_{W^{[m],s}_{G}})^{\otimes r}$. Then, again, as $W^{[m],s}_{G}$ has codimension 2 in $\cX_G^{[m],s}$, where $\cX_G^{[m],s}$ is the stable locus of $\cX^{[m]}$ quotiented by $G$, this section extends to an everywhere nonvanishing section on $K_{\cX_G^{[m],s}}^{\otimes r}$.

    \medskip
    Let $\mathfrak{Y}$ be any of the stacks listed in Proposition \ref{semistable proposition}. Any stable object $(Z,\cX)$ of $\mathfrak{Y}$ is such that $Z\in \cX_G^{[m],s}$ for some stability condition. And we have shown that for any such $Z$, there is a section of $K_{\cX_G^{[m],s}}^{\otimes r}$ which does not vanish at $Z$. Now, for any such stack $f\colon \mathfrak{Y}\to C$ that we constructed, $\mathfrak{Y}_0\coloneqq f^{*}(0)$ is a reduced principal divisor in the following sense. Let $S$ and $\xi$ be as above and write $\cY\coloneqq \xi^* H^{m,s}_{[A,B]}$ in $\mathfrak{Y}$, where $H^{m,s}_{[A,B]}$ is the stable locus of length $m$ zero-dimensional subschemes on $X[A,B]$. Let $\phi\colon \cY \to H^{m,s}_{[A,B]}$ be the corresponding strongly cartesian morphism. Then we have that $\cY_0\coloneqq (f\circ \phi)^*(0)$ is a reduced principal divisor in $\cY$. We therefore have that
    \[
    K_{\cY} + \cY_0
    \]
    is semi-ample and $f\colon\mathfrak{Y}\to C$ is a good minimal model.
\end{proof}

\begin{remark}
    In the above proof, we used the assumption that all blow-ups made in the construction of $X[A,B]$ were small, i.e.\ that the map $b\colon X[A,B] \to X\times_{\AA^1} \AA^{n+1}$ contracts no divisors. This assumption happens to be true in our case and is convenient, but not strictly necessary. Indeed, keeping with the notation of the above proof, we have a diagram
    \begin{figure}[H]
    \centering
    \begin{tikzcd}\centering
        \Tilde{W_*^m} \arrow[r, "b_{\diag}"] \arrow[d, "q_m"]  & W_*^m \arrow[r, "b^m"] & (X\times_{\AA^1}\AA^{n+1})^m \arrow[d] \\
        W_*^{[m]} \arrow[d, "q_G"] \arrow[rr] &&C\times_{\AA^1}\AA^{n+1} \arrow[d, "/G"] \\
        W^{[m]}_{*,G} \arrow[rr] && C,
    \end{tikzcd}
    \label{2-form_diag}
    \end{figure}
    \noindent where the morphisms in the diagram denote the restrictions of the previous morphisms of the same names to the $W_*$ locus. As $G$ is given as the natural torus action on the projective coordinates introduced by the blow-up map $b$, the morphism $q_G$ must contract exactly the divisors coming from the exceptional divisors of $b^m$ and no other divisors. By applying a generalised version of the Beauville argument once more, we may obtain the desired result.
\end{remark}

\begin{corollary}\label{trivial can bun corollary}
    Suppose $X\to C$ has trivial relative canonical bundle. Then the relative canonical bundle $K_{\mathfrak{Y}/C}$ is trivial, for any of the stacks $\mathfrak{Y}\to C$ listed in Proposition~\ref{semistable proposition}.
\end{corollary}

\begin{remark}
    Finally, we would like to remark that the fact that we are working with stacks and not schemes here allows us to produce semistable minimal models, as we do not see any quotient singularities. It is possible to construct good minimal models in the world of schemes using these expanded degeneration techniques, but these will no longer be semistable; they will be \emph{divisorial log terminal} (see \cite{dFKX} for a definition) as they will contain finite quotient singularities, by nature of the expanded degeneration machinery. This can be seen to happen already in \cite{GHH} and \cite{GHHZ}.
\end{remark}

\subsection{Symplectic structure}\label{minimality}

We will now assume that $X\to C$ is a type III good degeneration of K3 surfaces. In particular, this means that $K_{X/C}$ is trivial. By Corollary \ref{trivial can bun corollary}, this implies that $K_{\mathfrak{Y}/C}$ is trivial, for any of the stacks $\mathfrak{Y}\to C$ listed in Proposition \ref{semistable proposition}. This can be seen alternatively as following from the fact that the relative holomorphic symplectic logarithmic 2-form on $X\to C$ induces a relative holomorphic symplectic logarithmic 2-form on $\mathfrak{Y}\to C$.

\medskip
As the dual complex of $X_0$ is a triangulated sphere, we may choose a labelling of the vertices of this triangulation, such that each vertex is labelled by either $Y_1$, $Y_2$ or $Y_3$ and each triangle has exactly one of each vertices. All the constructions we made on the local family $\Spec k[x,y,z,t]\to \Spec k[t]$ can then be extended to a family $X\to C$ of K3 surfaces by making the following modification. Everywhere we blew up the pullback of $Y_i\times_{\AA^1} V(t_j)$ inside some scheme, we now blow up instead the pullback of $Y_{(i)}\times_{\AA^1} V(t_j)$, where $Y_{(i)}$ denotes the union of all $Y_i$ components in $X$.

\begin{proposition}
    The stacks listed in Proposition \ref{semistable proposition} are symplectic, in the sense that they carry a nowhere degenerate logarithmic 2-form.
\end{proposition}

\begin{proof}
    Let $f\colon \mathfrak{Y}\to C$ be one such construction. Since $f$ is semistable, this follows from Section 7.2 of \cite{GHHZ}. The logarithmic 2-form is defined with respect to the divisorial logarithmic structure given by the divisors $0$ in $ C$ and $\mathfrak{Y}_0 = f^*(0)$ in $ \mathfrak{Y}$.
\end{proof}

\printbibliography[
heading=bibintoc,
title={References}
] 

\end{document}